%% file: main.tex
\documentclass[11pt]{article}

\usepackage{amsmath}%
\usepackage{amsfonts}%
\usepackage{amssymb}%
\usepackage{graphicx}
\usepackage{mathtools}
\usepackage[T2A]{fontenc}
\usepackage[cp1251]{inputenc}
\usepackage{afterpage} 

\usepackage{subcaption}

\usepackage{xcolor}

\usepackage{amsthm} 
\usepackage[colorlinks]{hyperref} 
\AtBeginDocument{\hypersetup{
    citecolor=magenta,
    colorlinks=true,
    linkcolor=blue,
    filecolor=magenta,      
    urlcolor=cyan,    
    pdfpagemode=FullScreen,
    }}

\usepackage{enumerate}

\usepackage[all,cmtip]{xy}

\usepackage{doi}

\DeclareMathOperator{\rank}{rank}

\DeclareMathOperator{\GL}{GL}
\DeclareMathOperator{\Spec}{Spec}
\DeclareMathOperator{\SL}{SL}
\DeclareMathOperator{\bg}{\mathbf{\Gamma}}

\DeclareMathOperator{\Frac}{Frac}

\DeclareMathOperator{\htt}{ht}

\DeclareMathOperator{\Gr}{Gr}
\DeclareMathOperator{\Mat}{Mat}

\newcommand{\mc}{\mathcal{C}}

\newcommand{\std}{\text{std}}
\newcommand{\pprime}{{\prime\prime}}

\newtheorem{theorem}{Theorem}[section]
\newtheorem{lemma}[theorem]{Lemma}
\newtheorem{proposition}[theorem]{Proposition}
\newtheorem*{proposition*}{Proposition}

\theoremstyle{definition}\newtheorem{definition}[theorem]{Definition}
\theoremstyle{definition}\newtheorem{remark}[theorem]{Remark}
\theoremstyle{remark}
\numberwithin{equation}{section}

\title{Starfish lemma via birational quasi-isomorphisms}
\author{Dmitriy Voloshyn}
\date{}

\newcounter{myfootnote}
\begin{document}
\maketitle

\makeatletter
\def\blfootnote{\xdef\@thefnmark{}\@footnotetext}
\makeatother

\begin{abstract}
    We study birational quasi-isomorphisms between normal Noetherian domains endowed with cluster structures of geometric type. We prove an analogue of the Starfish lemma that allows one to transfer various cluster and algebraic properties of one variety onto another. In particular, we develop tools for proving that an upper cluster algebra equals the given commutative ring.
\end{abstract}

\tableofcontents

\blfootnote{\textit{2010 Mathematics Subject Classification.} 13F60.}
\blfootnote{\textit{Key words and phrases.} Cluster algebra, birational quasi-isomorphism, Starfish lemma.}

 \section{Introduction}
 \input{introduction}
 \section{Background on cluster algebras}\label{s:background}
 \input{background}
 \section{Definitions and basic properties}\label{s:defs}
 \input{defs}
 \section{Main results}\label{s:main_results}
 \input{main_results}
 \section{Single birational quasi-isomorphism}\label{s:singlebirat}
 \input{birat_single}
 \section{Complementary birational quasi-isomorphisms}\label{s:doublebirat}
 \input{birat_double}
\appendix
\section{Example in $\mathbb{C}[\SL_3]$}\label{s:appendix}
\input{example}
\section{Example in $\mathbb{C}[\widehat{\Gr}(3,7)]$}
\input{examplegr}\label{s:appendixgr}
 \input{biba}
\end{document}

%% file: introduction.tex
Cluster algebras were introduced by S. Fomin and A. Zelevinsky in \cite{fathers}. One of the initial objectives was to represent the coordinate ring of a complex variety $V$ (often arising in Lie theory) in the form $\mathcal{A}_{\mathbb{C}}(\mathcal{C}) = \mathbb{C}[V]$ where $\mathcal{C}$ is a cluster structure in $\mathbb{C}[V]$ and $\mathcal{A}_{\mathbb{C}}(\mathcal{C})$ is a cluster algebra. More generally, one aims to prove the equality $\bar{\mathcal{A}}_{\mathbb{C}}(\mathcal{C}) = \mathbb{C}[V]$ where $\bar{\mathcal{A}}_{\mathbb{C}}(\mathcal{C})$ is the upper cluster algebra.

Many examples of cluster algebras are known. Early examples include the Grassmannian~\cite{scott}, double Bruhat cells \cite{upper_bounds}, and simply connected simple complex algebraic groups \cite{conj}. More recent examples include open Richardson varieties \cite{open_rich} and Bott-Samelson cells~\cite{bott_shen}.

The main technique for proving the inclusion $\bar{\mathcal{A}}_{\mathbb{C}}(\mathcal{C}) \subseteq \mathbb{C}[V]$ (or $\mathcal{A}_{\mathbb{C}}(\mathcal{C}) \subseteq \mathbb{C}[V]$) is the so-called Starfish lemma (see Proposition~\ref{p:starfish}). The lemma requires verifying certain regularity and coprimality conditions in the initial cluster and its one-step mutations. To show the reverse inclusion $\bar{\mathcal{A}}_{\mathbb{C}}(\mathcal{C}) \supseteq \mathbb{C}[V]$, one typically applies a result on upper bounds (see Proposition~\ref{p:upper_bound}).

The work of M. Gekhtman, M. Shapiro and A. Vainshtein shows that the coordinate ring of a simply connected simple complex algebraic group $G$ can support several cluster structures of different ranks \cite{plethora,conj,exotic}. In general, it is harder to verify the conditions of the Starfish lemma in $\mathbb{C}[G]$ for cluster structures of higher ranks. However, as was shown in~\cite{plethora,rho,multdual,multdouble}, such cluster structures are typically related by a birational quasi-isomorphism.

In this paper, we axiomatize the notion of a birational quasi-isomorphism, and use it to prove an analogue of the Starfish lemma. Specifically, let $\tilde{\mathcal{C}}$ and $\mathcal{C}$ be cluster structures of geometric type. We let $\tilde{N}$ and $\tilde{M}$ (resp. $N$ and $M$) to be the rank and the number of frozen variables of $\tilde{\mathcal{C}}$ (resp. of $\mathcal{C}$), and we assume that $K:=N-\tilde{N} > 0$ and $\tilde{N}+\tilde{M} = N + M$. Let $(\tilde{x}_1,\ldots,\tilde{x}_{N+M})$ and $(x_1,\ldots,x_{N+M})$ be initial extended clusters in $\tilde{\mathcal{C}}$ and $\mathcal{C}$. We enumerate the variables in such a way that the first $\tilde{N}$ (resp. $N$) variables in the tuples are mutable. Since $K > 0$, the variables $(x_{\tilde{N}+1},\ldots,x_{\tilde{N}+K})$ (resp. $(\tilde{x}_{\tilde{N}+1},\ldots,\tilde{x}_{\tilde{N}+K})$ ) are mutable in $\mathcal{C}$ (resp. frozen in $\tilde{\mathcal{C}}$). We call such variables \emph{marked} (whether they are in $\mathcal{C}$ or in $\tilde{\mathcal{C}}$).

Freezing the marked variables in $\mathcal{C}$, assume there is a quasi-isomorphism $\mathcal{Q}$ between $\tilde{\mathcal{C}}$ and $\mathcal{C}$, in the sense of C. Fraser \cite{fraser}, with an additional restriction: a (cluster or frozen) variable $x_i$ in $\mathcal{C}$ is sent to the corresponding variable $\tilde{x}_i$ times a monomial \emph{in the marked variables} (a more precise definition is given in~\ref{def:quasiiso}). If $\tilde{\mathcal{C}}$ (resp. $\mathcal{C}$) is a cluster structure in a normal Noetherian domain $\tilde{R}$ (resp. $R$), we call $\mathcal{Q}$ a \emph{birational quasi-isomorphism} if $\mathcal{Q}$ extends to an isomorphism of the localizations of $R$ and $\tilde{R}$ with respect to the marked variables:
\begin{equation}
    \mathcal{Q} : R\left[x_{\tilde{N}+1}^{\pm 1},\ldots,x_{\tilde{N}+K}^{\pm 1}\right] \xrightarrow{\sim} \tilde{R}\left[\tilde{x}_{\tilde{N}+1}^{\pm 1},\ldots,\tilde{x}_{\tilde{N}+K}^{\pm 1}\right].
\end{equation}

Under the above setup, our first main result (Proposition~\ref{p:birat_single}) can be described as follows: If the conditions of the Starfish lemma are valid for $\tilde{\mathcal{C}}$ (which implies  $\bar{\mathcal{A}}(\tilde{\mathcal{C}})\subseteq \tilde{R}$), then, under some natural assumptions on the marked variables in $\mathcal{C}$, the conditions are also valid in $\mathcal{C}$ (in particular, $\bar{\mathcal{A}}(\mathcal{C}) \subseteq R$). We emphasize that $\mathcal{C}$ is of rank $N > \tilde{N}$ (hence mutations of the marked variables in $\mathcal{C}$ do not correspond to mutations in $\tilde{\mathcal{C}}$). With Proposition~\ref{p:upper_cont1}, one can also deduce (under an additional assumption on the marked variables) that, if $\bar{\mathcal{A}}(\tilde{\mathcal{C}}) = \tilde{\mathcal{C}}$, then $\bar{\mathcal{A}}({\mathcal{C}}) = {\mathcal{C}}$.

Our second main result (Proposition~\ref{p:birat_double}) concerns the case when there is yet another cluster structure $\hat{\mathcal{C}}$ in $\hat{R}$ along with a birational quasi-isomorphism $\mathcal{Q}^\prime$ between $\mathcal{C}$ and $\hat{\mathcal{C}}$. If the marked variables relative $(\hat{\mathcal{C}},\mathcal{C})$ are disjoint from the marked variables relative $(\tilde{\mathcal{C}},\mathcal{C})$, the assumptions on the marked variables are mild, and if $R$, $\tilde{R}$ and $\hat{R}$ are UFD's, the assumptions can be completely dropped.

The scope of applications covers (however, not entirely) the program of Gekhtman--Shapiro--Vainshtein on constructing (generalized) cluster structures compatible with Poisson brackets from the Belavin--Drinfeld class (for connections between cluster algebras and Poisson geometry, we refer to \cite{dasbuch}). The program was first formulated in~\cite{conj}. Some of birational quasi-isomorphisms for $\mathbb{C}[G]$ were described in \cite{plethora,rho}, and for the Poisson dual of $G$, in \cite{multdual}. We also know birational quasi-isomorphisms for the Drinfeld double (some of which were recorded in \cite{multdouble}) and for the Heisenberg double.

We expect that there are more examples in which birational quasi-isomorphisms arise. For instance, we expect that there are multiple cluster structures in $\mathbb{C}[\widehat{\Gr}(n,m)]$ (the affine cone of the Grassmannian) that are related via birational quasi-isomorphisms; one such example is given in Appendix~\ref{s:appendixgr}.

The paper is organized as follows. In Section~\ref{s:background}, we review the necessary background on cluster algebras. In Section~\ref{s:defs}, we define the notion of a birational quasi-isomorphism. In Section~\ref{s:main_results}, we state our main results. In Section~\ref{s:singlebirat} and Section~\ref{s:doublebirat}, we provide proofs of the main results along with less important results. In Appendices~\ref{s:appendix} and~\ref{s:appendixgr}, we illustrate our theory in $\mathbb{C}[\SL_3]$ and in $\mathbb{C}[\widehat{\Gr}(3,7)]$.

\paragraph{Acknowledgments.} The author would like to thank M. Gekhtman for reading the early draft of the paper, for comments and suggestions. This work was supported by the Institute for Basic Science (IBS-R003-D1).

%% file: background.tex
In this section, we provide a background on the theory of cluster algebras of geometric type. The main references are~\cite{fomin13,fomin6,dasbuch} and the original papers~\cite{upper_bounds,fathers}. For this section, let us fix a field $\mathcal{F}$ of rational functions in $N + M$ variables over $\mathbb{Q}$. We refer to $\mathcal{F}$ as the \emph{ambient field}.

\begin{definition}
    An \emph{extended seed} is a  pair $\Sigma:=(\mathbf{x},B)$ where $B:=(b_{ij} \ | \ i \in [1,N], \ j \in [1,N+M])$ is an integer skew-symmetrizable\footnote{That is, there exists a diagonal $N\times N$ matrix $D$ with positive entries such that $D \cdot B^{[1,N]}$ is skew-symmetric, where $B^{[1,N]}$ refers to the first $N$ columns of $B$.} matrix and $\mathbf{x}:=(x_1,x_2,\ldots,x_{N+M})$ is a set of elements of $\mathcal{F}$ such that $\mathcal{F} = \mathbb{Q}(x_1,x_2,\ldots,x_{N+M})$. The matrix $B$ is called an \emph{exchange matrix} and $\mathbf{x}$ is called an \emph{extended cluster}. The elements $x_1,x_2,\ldots,x_N$ are called \emph{cluster variables} and $x_{N+1},\ldots,x_{N+M}$ are called \emph{frozen variables}.
\end{definition}

\begin{definition}
    Let $\Sigma:=(\mathbf{x},B)$ be an extended seed. A \textit{mutation of $\Sigma$ in direction $k \in [1,N]$} is an extended seed $\Sigma^{\prime}_k:=(\mathbf{x}^\prime_k,B^\prime)$ such that:
    \begin{enumerate}[1)]
        \item The entries of $B^\prime := (b_{ij}^{\prime} \ | \ i\in[1,N], \ j \in [1,N+M])$ and the entries of $B$ are related via
        \begin{equation}\label{eq:b_mut}
            b_{ij}^\prime = \begin{cases}
                -b_{ij} \ &\text{if} \ i = k \ \text{or} \ j = k;\\
                b_{ij} + \frac{1}{2} (|b_{ik}|b_{kj} + b_{ik} |b_{kj}|) \ &\text{otherwise.}
            \end{cases}
        \end{equation} 
        \item 
        $\mathbf{x}_k^\prime := (\mathbf{x}\setminus\{x_k\})\cup \{x_k^\prime\}$ where $x_k^\prime$ satisfies the \emph{exchange relation}
        \begin{equation}\label{eq:ord_mut}
            x_kx_k^{\prime} = \prod_{b_{kj}>0} x_j^{b_{kj}} + \prod_{b_{kj}<0} x_j^{-b_{kj}}.
        \end{equation}
    \end{enumerate}
\end{definition}

\begin{definition}
    Two extended seeds $\Sigma$ and $\Sigma^\prime$ are called \emph{mutation equivalent} if there is a finite sequence of extended seeds $\Sigma_{i_0}^{\prime}, \Sigma_{i_1}^{\prime},\ldots, \Sigma_{i_m}^{\prime}$ such that $\Sigma_{i_0}^{\prime} = \Sigma$, $\Sigma_{i_m}^{\prime} = \Sigma^{\prime}$, and the extended seed $\Sigma_{i_{j}}^{\prime}$ is a mutation of $\Sigma_{i_{j-1}}^{\prime}$ in direction $i_{j}$, $j \geq 1$. For a given extended seed $\Sigma_0:=(\mathbf{x}_0,B_0)$, the collection of all extended seeds mutation equivalent to $\Sigma_0$ is called a \emph{cluster structure of geometric type}  and is denoted as $\mathcal{C}(\Sigma_0)$ or simply as $\mathcal{C}$. The extended seed $\Sigma_0$ is called the \emph{initial extended seed}, $\mathbf{x}_0$ is the \emph{initial extended cluster} and $B_0$ is the \emph{initial exchange matrix}.
\end{definition}

When several cluster structures are considered at once, the ambient field will be decorated with a subscript, such as $\mathcal{F}_{\mathcal{C}}$. The number $N$ is also called the \emph{rank} of $\mathcal{C}$.

\begin{definition}
    Given an extended seed $(\mathbf{x},B)$, $\mathbf{x} := (x_1,x_2,\ldots,x_{N+M})$, a \emph{$y$-seed} is a pair $(\mathbf{y},B)$ where $\mathbf{y}:=(y_1,y_2,\ldots,y_N)$, and the variables $y_i \in \mathcal{F}_{\mathcal{C}}$ are given by
\begin{equation}
    y_i := \prod_{j = 1}^{N+M} x_j^{b_{ij}}.
\end{equation}
The variables $y_i$ are called \emph{$y$-variables}, $i \in [1,N]$.
\end{definition}

The $y$-seeds have their own mutation patterns, which we do not need in this paper (for instance, see~\cite[Chapter 3]{fomin13}).

For an extended cluster $\mathbf{x}:=(x_1,x_2,\ldots,x_{N+M})$, define 
\begin{equation}
\mathcal{A}(\mathbf{x}):=\mathbb{Z}[x_1,x_2,\ldots,x_{N+M}], \ \ \mathcal{L}(\mathbf{x}) := \mathbb{Z}[x_1^{\pm 1},x_2^{\pm 1},\ldots,x_N^{\pm 1},x_{N+1},\ldots,x_{N+M}].
\end{equation}
\begin{definition}
    Given a cluster structure $\mathcal{C}$, a \emph{cluster algebra} $\mathcal{A}(\mathcal{C})$ is a $\mathbb{Z}$-subalgebra of $\mathcal{F}_{\mathcal{C}}$ generated by all cluster and frozen variables in $\mathcal{C}$. The \emph{upper cluster algebra} $\bar{\mathcal{A}}(\mathcal{C})$ is an algebra given by
    \begin{equation}
        \bar{\mathcal{A}}(\mathcal{C}) := \bigcap_{\mathbf{x}\in\mathcal{C}} \mathcal{L}(\mathbf{x}).
    \end{equation}
\end{definition}
One of the major results in cluster theory is the \emph{Laurent phenomenon}~\cite[Theorem 3.1]{fathers}; that is, $\mathcal{A}(\mathcal{C})\subseteq \bar{\mathcal{A}}(\mathcal{C})$. The following concept of an upper bound was introduced in~\cite{upper_bounds}:

\begin{definition}
    Let $\mathbb{T}_N$ be an $N$-regular tree. Associate with each vertex an extended seed in such a way that two vertices are adjacent if and only if the corresponding extended seeds are obtained from one another via a mutation in direction $i$ for some $i\in[1,N]$; label the corresponding edge with $i$. A \emph{nerve} $\mathcal{N}$ in $\mathbb{T}_N$ is a subtree on $N+1$ vertices such that all the edges in $\mathcal{N}$ have different labels. An \emph{upper bound} $\bar{\mathcal{A}}(\mathcal{N})$ is the $\mathbb{Z}$-subalgebra of $\mathcal{F}$ given by
    \begin{equation}
        \bar{\mathcal{A}}(\mathcal{N}) := \bigcap_{\mathbf{x}\in \mathcal{N}} \mathcal{L}(\mathbf{x}).
    \end{equation}
\end{definition}
The following result corresponds to \cite[Corollary 1.9]{upper_bounds}. 

\begin{proposition}\label{p:upper_bound}
    Assume that the initial exchange matrix in $\mathcal{C}$ has full rank. Then the upper bounds $\bar{\mathcal{A}}(\mathcal{N})$ do not depend on the choice of the nerve $\mathcal{N}$ and $\bar{\mathcal{A}}(\mathcal{N}) = \bar{\mathcal{A}}(\mathcal{C})$.
\end{proposition}

We now discuss cluster structures in integral domains. For a given integral domain $R$, we denote by $\Frac R$ its field of fractions.

\begin{definition}\label{d:clustring}
    Let $R$ be an integral domain and $\imath : \mathcal{A}(\mathbf{x}_0) \rightarrow \Frac R$ be an embedding. The triple $(\mathcal{C},R,\imath)$ is called a \emph{cluster structure in $R$}.
\end{definition}
Note that the embedding $\imath : \mathcal{A}(\mathbf{x}_0) \rightarrow \Frac R$ induces an embedding of the ambient field $\mathcal{F}$ into $\Frac R$ (thus all cluster variables are realized as elements of $\Frac R$). We will frequently identify $\mathcal{F}$ with a subfield of $\Frac R$ and omit $\imath$.

\begin{definition}\label{d:regular}
    Given a cluster structure $\mathcal{C}$ in an integral domain $R$, a cluster or frozen variable $x$ in $\mathcal{C}$ is called \emph{regular} if $x \in R$; an extended cluster $\mathbf{x}$ is called \emph{regular} if for each $x \in \mathbf{x}$, $x$ is regular; and finally, $(\mathcal{C},R)$ is called \emph{regular} if each extended cluster in $\mathcal{C}$ is regular.
\end{definition}

We call an integral domain $R$ \emph{normal} if it is integrally closed in $\Frac R$. Two elements $x,y\in R$ are called \emph{coprime} if they are not contained in the same prime ideal of height~$1$. The following proposition was proved in~\cite[Proposition~3.6]{tensordiags} (see also~\cite[Proposition 6.4.1]{fomin6}; for arbitrary nerves, see \cite[Proposition 9.1]{fraser}).

\begin{proposition}[\emph{Starfish Lemma}]\label{p:starfish} Let $R$ be a normal Noetherian domain and $\mathcal{C}:=\mathcal{C}(\mathbf{x}_0,B_0)$ be a cluster structure of geometric type in $R$. Assume the following:
\begin{enumerate}[1)]
\item All variables in $\mathbf{x}_0$ are regular;
\item The cluster variables in $\mathbf{x}_0$ are pairwise coprime;
\item For each cluster variable $x$ in $\mathbf{x}_0$, $x^\prime$ is regular and coprime with $x$.
\end{enumerate}
Then $\bar{\mathcal{A}}(\mathcal{C}) \subseteq R$. 
\end{proposition}

\paragraph{Toric actions.} Let $\mathcal{C}$ be a cluster structure of geometric type, $\mathcal{F}$ be its ambient field and $\mathbf{x} \in \mathcal{C}$ be an extended cluster. For a field $\mathbb{K}\supseteq \mathbb{Q}$, set $\mathcal{F}_{\mathbb{K}}:=\mathcal{F}\otimes \mathbb{K}$. For an integer matrix $W:=(w_{ij})$ of size $(N+M)\times s$, where $s>0$, consider an action of $(\mathbb{K}^{\times})^s$ upon $\mathcal{F}_{\mathbb{K}}$ by field automorphisms: for every $\mathbf{t}:=(t_1,t_2,\ldots,t_s) \in (\mathbb{K}^{\times})^{s}$, the corresponding field automorphism $\sigma_{\mathbf{t}}$ is given by
\begin{equation}\label{eq:tor_loc}
    \sigma_{\mathbf{t}}(x_i) = \left(\prod_{j=1}^{s}t_j^{w_{ij}}\right) \cdot x_i, \ \ i \in [1,N+M].
\end{equation}
\begin{definition}
    A \emph{local toric action} of rank $s$ upon an extended cluster $\mathbf{x}$ is the action by field automorphisms given by equation~\eqref{eq:tor_loc} such that $W$ is of rank $s$. It is called a \emph{global toric action} if the action is equivariant\footnote{In other words, given any other extended cluster $\mathbf{x}^\prime$ mutation equivalent to $\mathbf{x}$, there exists a matrix $W^\prime$ such that the action is still of the form~\eqref{eq:tor_loc}.} with respect to mutations. 
\end{definition}

The following proposition was proved in~\cite[Lemma 2.3]{roots}.

\begin{proposition}\label{p:torbw}
    Let $(\mathbf{x},B)$ be an extended seed in $\mathcal{C}$. Suppose there is a local toric action of rank $s$ upon $\mathbf{x}$ given by a weight matrix $W$. Then the action is global if and only if $BW = 0$.
\end{proposition}

%% file: defs.tex
In this section, we state definitions and main properties of related cluster structures and quasi-isomorphisms.

\subsection{Related cluster structures} Let $\mathcal{C}$ and $\tilde{\mathcal{C}}$ be cluster structures of geometric type of ranks $N$ and $\tilde{N}$ with $M$ and $\tilde{M}$ frozen variables, respectively. Assume that $\tilde{N} < N$ and $N+M = \tilde{N} + \tilde{M}$. Fix initial extended seeds $\Sigma_0 := (\mathbf{x}_0,B_0)$ and $\tilde{\Sigma}_0 := (\tilde{\mathbf{x}}_0,\tilde{B}_0)$ in $\mathcal{C}$ and $\tilde{\mathcal{C}}$, respectively, where $\mathbf{x}_0:=(x_1,x_2,\ldots,x_{N+M})$ and $\tilde{\mathbf{x}}_0:=(\tilde{x}_1,\tilde{x}_2,\ldots,\tilde{x}_{N+M})$. Let $\kappa$ be a permutation of the set of indices $[1,N+M]$. Set
\begin{align}
&\tilde{\mathcal{I}}(\kappa) := [\tilde{N}+1,\tilde{N} + \tilde{M}] \cap \kappa([1,N]);\\
&\mathcal{I}(\kappa):=\kappa^{-1}(\tilde{\mathcal{I}}(\kappa)) = [1,N] \cap \kappa^{-1}([\tilde{N}+1,\tilde{N} + \tilde{M}]).
\end{align}

\begin{definition}
A \emph{marker} $\kappa$ for the pair $(\mathcal{C},\tilde{\mathcal{C}})$ is a bijection $k:[1,N+M]\rightarrow[1,N+M]$ such that $\kappa([N+1,N+M]) \subset [\tilde{N}+1,\tilde{N}+\tilde{M}]$. For a given marker $\kappa$, the triple $(\mathcal{C},\tilde{\mathcal{C}},\kappa)$ is called a \emph{related} triple. The sets $\mathcal{I}(\kappa)$ and $\tilde{\mathcal{I}}(\kappa)$ are called the sets of \emph{marked indices} in $\mathcal{C}$ and $\tilde{\mathcal{C}}$, respectively.
\end{definition}
 Assume that $(\mathcal{C},\tilde{\mathcal{C}},\kappa)$ is a related triple. Define
\begin{align}
&\mathbf{x}_0(\kappa):=\{x_i \in \mathbf{x}_0 \ | \ i \in \mathcal{I}(\kappa)\},\\
&\tilde{\mathbf{x}}_0(\kappa):=\{\tilde{x}_i \in \tilde{\mathbf{x}}_0 \ | \ i \in \tilde{\mathcal{I}}(\kappa)\}.\label{eq:tildx0kap}
\end{align}
\begin{definition}\label{d:marked_vars}
The elements of $\mathbf{x}_0(\kappa)$ and $\tilde{\mathbf{x}}_0(\kappa)$ are called the \emph{marked variables}\footnote{We use the same terminology for marked variables in $\mathcal{C}$ and in $\tilde{\mathcal{C}}$. It should be clear from the context and notation (whether there is a tilde or not) where a given marked variable belongs to.} in $\mathcal{C}$ and $\tilde{\mathcal{C}}$.
\end{definition}

\begin{definition}
An extended seed $\Sigma:=(\mathbf{x},B)$ in $\mathcal{C}$ is called \emph{marked} if it is obtained from $\Sigma_0$ via a sequence of mutations 
\begin{equation}\label{eq:emarkcl}
    \Sigma_{i_0} \rightarrow \Sigma_{i_1} \rightarrow \cdots \rightarrow \Sigma_{i_m}
\end{equation}
where $\Sigma_{i_0}:=\Sigma_0$, $\Sigma_{i_m}:=\Sigma$, and $\Sigma_{i_j}$ is a mutation of $\Sigma_{i_{j-1}}$ in direction $i_j \in [1,N] \setminus \mathcal{I}(\kappa)$, $j \in [1,m]$. In this case, the extended cluster $\mathbf{x}$ is called \emph{marked}.
\end{definition}

In particular, the initial extended seed $\Sigma_0$ is marked. A marker $\kappa$ induces a natural bijection between extended marked seeds in $\mathcal{C}$ and all extended seeds in $\tilde{\mathcal{C}}$, as follows. If an extended seed $\Sigma$ is obtained from $\Sigma_0$ via the sequence~\eqref{eq:emarkcl}, consider the sequence of mutations in $\tilde{\mathcal{C}}$
\begin{equation}
\tilde{\Sigma}_{\kappa(i_0)} \rightarrow \tilde{\Sigma}_{\kappa(i_1)} \rightarrow \cdots \rightarrow \tilde{\Sigma}_{\kappa(i_m)}
\end{equation}
where $\tilde{\Sigma}_{\kappa(i_0)}:=\tilde{\Sigma}_0$. Then we set $\tilde{\Sigma}:=\kappa(\Sigma):=\tilde{\Sigma}_{\kappa(i_m)}$.

\begin{definition}
Extended seeds $\Sigma:=(\mathbf{x},B) \in \mathcal{C}$ and $\tilde{\Sigma}:=(\tilde{\mathbf{x}},\tilde{B}) \in \tilde{\mathcal{C}}$ are called \emph{related} if $\tilde{\Sigma} = \kappa(\Sigma)$. In this case, the extended clusters $\mathbf{x}$ and $\tilde{\mathbf{x}}$ are called \emph{related}, and for any $i \in [1,N+M]$, the variables $x_i \in \mathbf{x}$ and $\tilde{x}_{\kappa(i)}\in \tilde{\mathbf{x}}$ are called \emph{related}.
\end{definition}

By default, if $x$ denotes a variable from an extended marked cluster in $\mathcal{C}$, $\tilde{x}$ denotes the corresponding related variable in $\tilde{\mathcal{C}}$.

\subsection{Quasi-isomorphisms} In this subsection, we introduce the notion of a quasi-isomorphism in the presence of marked variables, as well as some basic results. We continue in the setup of the previous subsection and consider a related triple $(\mathcal{C},\tilde{\mathcal{C}},\kappa)$ with fixed initial extended seeds $\Sigma_0:=(\mathbf{x}_0,B_0)$ and $\tilde{\Sigma}_0:=(\tilde{\mathbf{x}}_0,\tilde{B}_0)$.

\begin{definition}\label{def:quasiiso}
A \emph{quasi-isomorphism} $\mathcal{Q}:\mathcal{C} \rightarrow \tilde{\mathcal{C}}$ is 
a field isomorphism $\mathcal{Q}:\mathcal{F}_{\mathcal{C}}\rightarrow \mathcal{F}_{\tilde{\mathcal{C}}}$ such that for any extended marked cluster $\mathbf{x}$, there exists an integer matrix
\begin{equation}
\Lambda(\mathbf{x}):=\left(\lambda_{ij} \ | \ i\in[1,N+M], \ j \in {\mathcal{I}}(\kappa)\right)
\end{equation} 
such that
\begin{equation}\label{eq:quasi_formula}
\mathcal{Q}(x_i) = \begin{cases}
    \displaystyle\tilde{x}_{\kappa(i)} \prod_{j \in {\mathcal{I}}(\kappa)} \tilde{x}_{\kappa(j)}^{\lambda_{ij}} & i \in [1,N+M]\setminus \mathcal{I}(\kappa); \\[10pt] \displaystyle\prod_{j \in {\mathcal{I}}(\kappa)} \tilde{x}_{\kappa(j)}^{\lambda_{ij}} & i \in \mathcal{I}(\kappa).
\end{cases}
\end{equation}
\end{definition}

\begin{remark}
    If one freezes the marked variables in $\mc$, then $\mathcal{Q}:{\mathcal{C}}\rightarrow {\tilde{\mathcal{C}}}$ is a particular instance of the quasi-isomorphism defined in~\cite{fraser}. We emphasize that the quasi-isomorphism depends on the marker $\kappa$. It might happen that cluster structures are quasi-isomorphic for a different choice of $\kappa$.
\end{remark}


For an exchange matrix $\tilde{B}$ in $\tilde{\mathcal{C}}$, let $P$ be a permutation matrix such that $P^{T} \tilde{B}^{[1,\tilde{N}]} P$ is block diagonal, with blocks given by $D_1, D_2,\ldots, D_m$. We say that an integer matrix $\tilde{B}^\prime$ of size $\tilde{N}\times (\tilde{N}+\tilde{M})$ is obtained from $\tilde{B}$ by \emph{flipping orientations} if $P^{T} [\tilde{B}^{\prime}]^{[1,\tilde{N}]} P$ is block diagonal, with blocks given by $\epsilon_1 D_1, \epsilon_2 D_2,\ldots, \epsilon_n D_n$ for some choices $\epsilon_i \in \{-1,1\}$. 

For an $\tilde{N} \times (\tilde{N}+\tilde{M})$ matrix $A$, denote by $\kappa^{-1}(A)$ the matrix given by $(\kappa^{-1}(A))_{i,j} := A_{\kappa(i),\kappa(j)}$, $i \in [1,N] \setminus \mathcal{I}(\kappa)$, $j \in [1,N+M]$. The next result is an adaptation of Corollary~4.5 from~\cite{fraser} to our definition of a quasi-isomorphism.

\begin{proposition}\label{p:defquasi}
    Let $\mathcal{Q} : \mathcal{F}_{\mathcal{C}} \rightarrow \mathcal{F}_{\tilde{\mathcal{C}}}$ be a field isomorphism given by equation~\eqref{eq:quasi_formula} for some extended seed $(\mathbf{x},B)$. TFAE:
    \begin{enumerate}[i)]
        \item For $i \in [1,N] \setminus \mathcal{I}(\kappa)$ and the $y$-variables corresponding to $\mathbf{x}$, $\mathcal{Q}(y_i) = \tilde{y}_{\kappa(i)}$;\label{i:quasi1}
        \item The exchange matrix $B$ and the matrix $\Lambda$ satisfy the following equations:
        \begin{equation}\label{eq:qextrapol}
        [\kappa^{-1}(\tilde{B})]_{[1,N]\setminus\mathcal{I}(\kappa)}^{[1,N+M]\setminus\mathcal{I}(\kappa)} = B^{[1,N+M]\setminus\mathcal{I}(\kappa)}_{[1,N]\setminus\mathcal{I}(\kappa)}\ \ \ \text{and} \ \ \ [\kappa^{-1}(\tilde{B})]_{[1,N]\setminus\mathcal{I}(\kappa)}^{\mathcal{I}(\kappa)} = B_{[1,N]\setminus\mathcal{I}(\kappa)}\cdot \Lambda.
        \end{equation}\label{i:quasi2}
    \end{enumerate}
    Moreover, if Condition~\ref{i:quasi1} or~\ref{i:quasi2} is satisfied, then $\mathcal{Q}$ is a quasi-isomorphism. Conversely, if $\mathcal{Q}$ is a quasi-isomorphism, then there exists a flip of orientations of $\tilde{B}$ such that Conditions~\ref{i:quasi1}-\ref{i:quasi2} are satisfied.
\end{proposition}

\begin{proposition}\label{p:quasi_toric}
    Suppose that $\mathcal{I}(\kappa) = \{m\}$ for some $m \in [1,N]$ and that $\mathcal{C}$ and $\tilde{\mathcal{C}}$ admit global toric actions of rank $1$. Let $w$ and $\tilde{w}$ be the weight vectors of the toric actions in $\mathbf{x}_0$ and $\tilde{\mathbf{x}}_0$. Set
    \begin{equation}
        \Lambda_i(\mathbf{x}_0):=\Lambda_i:=\begin{cases}
            \dfrac{w_i-\tilde{w}_{\kappa(i)}}{\tilde{w}_{\kappa(m)}} \ &\text{if} \ i \neq m\\[3pt]
            1 \ &\text{if}\ i = m,
        \end{cases}
    \end{equation}
    and assume that $\Lambda_i \in \mathbb{Z}$ for each $i \in [1,N+M]$, $w_{m} = \tilde{w}_{\kappa(m)}$ and \begin{equation}[\kappa^{-1}(\tilde{B})]_{[1,N]\setminus\{m\}}^{[1,N+M]\setminus\{m\}} = B^{[1,N+M]\setminus\{m\}}_{[1,N]\setminus\{m\}}.\end{equation}
    Then $\Lambda$ defines a quasi-isomorphism $\mathcal{Q}:\mathcal{C}\rightarrow\tilde{\mathcal{C}}$. Moreover, if $(\mathbf{x},\tilde{\mathbf{x}})$ is a pair of related extended clusters and $w^\prime$ and $\tilde{w}^\prime$ are the corresponding weight vectors of the toric actions, then the corresponding matrix $\Lambda(\mathbf{x})$ is given by
    \begin{equation}
        \Lambda_i(\mathbf{x})=\begin{cases}
            \dfrac{w_i^\prime-\tilde{w}_{\kappa(i)}^\prime}{\tilde{w}_{\kappa(m)}} \ &\text{if} \ i \neq m\\[3pt]
            1 \ &\text{if}\ i = m.
        \end{cases}
    \end{equation}
\end{proposition}

\begin{remark}
The earliest version of Proposition~\ref{p:quasi_toric} appeared in~\cite[Lemma 8.4]{exotic}. The quasi-isomorphisms $\mathcal{U}^*$ in~\cite{plethora,multdouble} and the quasi-isomorphisms $\mathcal{Q}^*$ and $\mathcal{G}^*$ in~\cite{multdual} arise from global toric actions whose weight vectors are given by the polynomial degrees of the cluster and frozen variables.
\end{remark}


\begin{remark}\label{r:extension}
    Let us mention adjustments in the case of generalized cluster structures, as defined in~\cite{double}. Each cluster variable $x_i$ is endowed with additional data $(d_i,\{p_{ir}\}_{r=0}^{d_i},\{\hat{p}_{ir}\}_{r=0}^{d_i})$ where $d_i \geq 1$ is an integer, $p_{ir}$ is a monomial in the frozen variables for each $r$, $p_{i0} = p_{id_i} = 1$, and $\hat{p}_{ir}$ is a Laurent monomial in the frozen variables. The mutation relation for $x_i$ has $d_i+1$ terms in the right-hand side, and $\{p_{ir}\}_{r=0}^{d_i}$ or $\{\hat{p}_{ir}\}_{r=0}^{d_i}$ serve as coefficients\footnote{There are two equivalent ways to write the mutation relation, either via $p_{ir}$ or via $\hat{p}_{ir}$; for details, see~\cite{double}.} of the terms. Let $\mathcal{C}$ and $\tilde{\mathcal{C}}$ be generalized cluster structures. As above, all the data associated with $\tilde{\mathcal{C}}$ is decorated with a tilde.
    \begin{enumerate}[1)]
        \item For a related triple $(\mathcal{C},\tilde{\mathcal{C}},\kappa)$ of generalized cluster structures, we additionally impose $d_i = \tilde{d}_{\kappa(i)}$;\label{i:cond_mult}
        \item For quasi-isomorphisms, we additionally impose\footnote{In applications in Poisson geometry, $\tilde{\hat{p}}_{ir}$ are Casimirs of the Poisson bracket, and since quasi-isomorphisms are also expected to be Poisson isomorphisms (on open dense sets), it is a natural assumption.} $\mathcal{Q}(\hat{p}_{ir}) = \tilde{\hat{p}}_{\kappa(i)r}$.\label{i:cond_cas}
    \end{enumerate}
    With the above assumptions, Proposition~\ref{p:defquasi} holds true.  In Proposition~\ref{p:torbw}, the 'if' part holds if one additionally requires the condition
    \begin{enumerate}[1)]
    \setcounter{enumi}{2}
        \item\label{i:torinvc} The Laurent monomials $\hat{p}_{ir}$ are invariant with respect to the toric action for all $i$ and $r$.
    \end{enumerate} Proposition~\ref{p:upper_bound} holds under the following additional assumption:
    \begin{enumerate}[1)]
        \setcounter{enumi}{3}
        \item\label{i:gcs_upperb} Let $L$ be the number of isolated frozen variables in $\mathcal{C}$ and $P_i$ be a $(d_i-1)\times L$ matrix that consists of the exponents of the isolated frozen variables in $p_{ir}$. Then the condition is, for each $i \in [1,N]$, $\rank P_i = d_i - 1$.
    \end{enumerate}
    Lastly, to make sure that the quasi-isomorphism $\mathcal{Q}$ induced by toric actions in Proposition~\ref{p:quasi_toric} satisfies Condition~\ref{i:cond_cas}, one additionally imposes Condition~\ref{i:torinvc} and the condition
    \begin{enumerate}[1)]
    \setcounter{enumi}{3}
    \item For each $i \in [1,N]\setminus\mathcal{I}(\kappa)$ and each $1 \leq r \leq d_i$, $\kappa(\hat{p}_{ir})=\tilde{\hat{p}}_{\kappa(i)r}$
    \end{enumerate}
    where $\kappa$ is extended to a field isomorphism $\kappa:\mathcal{F}_{\mathcal{C}}\rightarrow \mathcal{F}_{\tilde{\mathcal{C}}}$ via $\kappa(x_i) = \tilde{x}_{\kappa(i)}$. We continue the discussion in Remark~\ref{r:gcs_extend} and indicate further adjustments to our main results.
\end{remark}

\subsection{Birational quasi-isomorphisms}
We continue in the setup of the previous section. We assume further that $\mathcal{C}$ and $\tilde{\mathcal{C}}$ are cluster structures in normal Noetherian domains $R$ and $\tilde{R}$, respectively.

\begin{definition}\label{d:birat}
A quasi-isomorphism $\mathcal{Q}:\mathcal{C}\rightarrow\tilde{\mathcal{C}}$ is called \emph{birational} if the marked variables $\mathbf{x}_0(\kappa)$ and $\tilde{\mathbf{x}}_0(\kappa)$ are regular, and the isomorphism\footnote{By $\mathbf{x}_0(\kappa)^{\pm 1}$ we mean the localization at each marked variable $x_i \in \mathbf{x}_0(\kappa)$.}
\begin{equation}
\mathcal{Q}:\mathcal{A}(\mathbf{x}_0)[\mathbf{x}_0(\kappa)^{\pm 1}] \xrightarrow{\sim} \mathcal{A}(\tilde{\mathbf{x}}_0)[\tilde{\mathbf{x}}_0(\kappa)^{\pm 1}]
\end{equation}
extends to an isomorphism of rings $\mathcal{Q}:R[\mathbf{x}_0(\kappa)^{\pm 1}] \xrightarrow{\sim} \tilde{R}[\tilde{\mathbf{x}}_0(\kappa)^{\pm 1}]$.
\end{definition} 

Assume there is another cluster structure $\hat{\mathcal{C}}$ in a normal Noetherian domain $\hat{R}$, and such that $(\mathcal{C},\hat{\mathcal{C}},\kappa^\prime)$ is a related triple. Similarly to the case of $\tilde{\mathcal{C}}$, all the data associated with $\hat{\mathcal{C}}$ is decorated with a hat. 

\begin{definition}\label{d:compl}
    Assume there is a pair $\mathcal{Q}:(\mathcal{C},R)\rightarrow(\tilde{\mathcal{C}},\tilde{R})$ and $\mathcal{Q}^\prime:(\mathcal{C},R) \rightarrow (\hat{\mathcal{C}},\hat{R})$ of birational quasi-isomorphisms. Then the pair is called \emph{complementary} if $\mathcal{I}(\kappa) \cap \mathcal{I}(\kappa^\prime) = \emptyset$.
\end{definition}

%% file: main_results.tex
In this section, we state the main results of the paper. We begin with a technical definition.

\begin{definition}\label{d:admissible}
    Let $A \subseteq \mathbf{x}_0$ be a subset of initial variables. We say that $A$ is \emph{admissible} if it satisfies the following conditions:
    \begin{enumerate}[1)]
        \item All variables in $A$ are regular;
        \item All cluster variables in $A$ are pairwise coprime;
        \item For each cluster variable $x \in A$, $x^\prime$ is regular and coprime with $x$.
    \end{enumerate}
\end{definition}

Likewise, one defines an admissible set of initial variables $A \subseteq \tilde{\mathbf{x}}_0$ in $(\tilde{\mathcal{C}},\tilde{R})$. Note that for $A = \mathbf{x}_0$, the conditions listed above are precisely the conditions of the Starfish lemma (see Proposition~\ref{p:starfish}).

\begin{proposition}\label{p:birat_single}
     Let $(\mathcal{C},R)$ and $(\tilde{\mathcal{C}},\tilde{R})$ be cluster structures of geometric type in normal Noetherian domains $R$ and $\tilde{R}$ with fixed initial extended clusters $\mathbf{x}_0$ and $\tilde{\mathbf{x}}_0$. Assume the following:
     \begin{enumerate}[1)]
         \item There is a marker $\kappa$ that makes $(\mathcal{C},\tilde{\mathcal{C}},\kappa)$ a related triple, along with a birational quasi-isomorphism $\mathcal{Q}:(\mathcal{C},R)\rightarrow(\tilde{\mathcal{C}},\tilde{R})$;\label{i:birsing1}
         \item The initial extended cluster $\tilde{\mathbf{x}}_0$ is admissible;\label{i:birsing2}
         \item The set of marked variables $\mathbf{x}_0(\kappa)$ is admissible;\label{i:birsing3}
         \item The initial extended cluster $\mathbf{x}_0$ is regular;\label{i:birsing4}
         \item Each nonmarked cluster variable in $\mathbf{x}_0$ is coprime with each marked variable. \label{i:birsing5}
    \end{enumerate}
    Then $\mathbf{x}_0$ is admissible, $\bar{\mathcal{A}}(\tilde{\mathcal{C}}) \subseteq \tilde{R}$ and $\bar{\mathcal{A}}(\mathcal{C}) \subseteq R$.
\end{proposition}

\begin{proposition}\label{p:birat_double}
    Let $R$, $\tilde{R}$ and $\hat{R}$ be normal Noetherian domains, $(\mathcal{C},R)$, $(\tilde{\mathcal{C}},\tilde{R})$ and $(\hat{\mathcal{C}},\hat{R})$ be cluster structures of geometric type in $R$, $\tilde{R}$ and $\hat{R}$, with fixed initial extended clusters $\mathbf{x}_0$, $\tilde{\mathbf{x}}_0$ and $\hat{\mathbf{x}}_0$. Assume the following:
    \begin{enumerate}[1)]
        \item There are markers $\kappa$ and $\kappa^\prime$ that make $(\mathcal{C},\tilde{\mathcal{C}},\kappa)$ and $(\mathcal{C},\hat{\mathcal{C}},\kappa^\prime)$ related triples, as well as a pair of complementary birational quasi-isomorphisms $\mathcal{Q}:(\mathcal{C},R) \rightarrow (\tilde{\mathcal{C}},\tilde{R})$ and $\mathcal{Q}^\prime:(\mathcal{C},R)\rightarrow(\hat{\mathcal{C}},\hat{R})$; \label{i:bird_1}
        \item The initial extended clusters $\tilde{\mathbf{x}}_0$ and $\hat{\mathbf{x}}_0$ are admissible; \label{i:bird_2}
        \item Every $x_i \in \mathbf{x}_0(\kappa)$ is coprime with every $x_j \in \mathbf{x}_0(\kappa^\prime)$; \label{i:bird_3}
        \item The initial extended cluster $\mathbf{x}_0$ is regular. \label{i:bird_4}
    \end{enumerate}
    Then $\mathbf{x}_0$ is admissible, as well as $\bar{\mathcal{A}}(\hat{\mathcal{C}}) \subseteq \hat{R}$, $\bar{\mathcal{A}}(\tilde{\mathcal{C}}) \subseteq \tilde{R}$ and $\bar{\mathcal{A}}(\mathcal{C}) \subseteq R$.
\end{proposition}

The next two results are analogues of Proposition~\ref{p:birat_single} and Proposition~\ref{p:birat_double} in case the given rings are UFD's. It is often desirable to show that the initial variables are irreducible\footnote{If $R$ is a UFD, we call a non-zero, non-invertible element $x \in R$ \emph{irreducible} if whenever $x = yz$ for some $y,z \in R$, then $y$ or $z$ is invertible. Recall that for UFD's, $x$ is irreducible if and only if $x$ is prime.}; this allows, in turn, to drop Condition~\ref{i:bird_3} in the previous proposition.

\begin{definition}\label{d:iadmissible}
    Let $A$ be a subset of initial variables. We say that $A$ is \emph{i-admissible} if it satisfies the following conditions:
    \begin{enumerate}[1)]
        \item All variables in $A$ are regular and irreducible;
        \item For each cluster variable $x_i \in A$, $x_i^\prime$ is regular and coprime with $x_i$.
    \end{enumerate}
\end{definition}
Likewise, we define an i-admissible set of initial variables in $(\tilde{\mathcal{C}},\tilde{R})$. 

We denote by $R^{\times}$ the group of units of $R$, and by $\mathbb{K}$ any field extension of $\mathbb{Q}$. By a factorial algebra we mean an algebra whose underlying commutative ring is a UFD.

\begin{proposition}\label{p:fbirat_single}
   Let $(\mathcal{C},R)$ and $(\tilde{\mathcal{C}},\tilde{R})$ be cluster structures of geometric type in factorial Noetherian $\mathbb{K}$-algebras with no zero divisors and such that $R^{\times} = \tilde{R}^\times = \mathbb{K}^{\times}$. Let $\mathbf{x}_0$ and $\tilde{\mathbf{x}}_0$ be fixed initial extended clusters. Assume the following:
     \begin{enumerate}[1)]
         \item There is a marker $\kappa$ that makes $(\mathcal{C},\tilde{\mathcal{C}},\kappa)$ a related triple, along with a birational quasi-isomorphism $\mathcal{Q}:(\mathcal{C},R)\rightarrow(\tilde{\mathcal{C}},\tilde{R})$;\label{i:fbirsing1}
         \item The initial extended cluster $\tilde{\mathbf{x}}_0$ is i-admissible;\label{i:fbirsing2}
         \item The set of marked variables $\mathbf{x}_0(\kappa)$ is i-admissible;\label{i:fbirsing3}
         \item The initial extended cluster $\mathbf{x}_0$ is regular;\label{i:fbirsing4}
        \item Each nonmarked cluster variable in $\mathbf{x}_0$ is coprime with each marked variable.\label{i:fbirsing5}
    \end{enumerate}
    Then $\mathbf{x}_0$ is i-admissible, as well as $\bar{\mathcal{A}}_{\mathbb{K}}(\tilde{\mathcal{C}}) \subseteq \tilde{R}$ and $\bar{\mathcal{A}}_{\mathbb{K}}(\mathcal{C}) \subseteq R$.
\end{proposition}

\begin{proposition}\label{p:birat_double_ufd}
    Let $R$, $\tilde{R}$ and $\hat{R}$ be factorial Noetherian $\mathbb{K}$-algebras with no zero divisors and such that $R^{\times} = \tilde{R}^\times = \hat{R}^{\times} = \mathbb{K}^{\times}$, $(\mathcal{C},R)$, $(\tilde{\mathcal{C}},\tilde{R})$ and $(\hat{\mathcal{C}},\hat{R})$ be cluster structures of geometric type in $R$, $\tilde{R}$ and $\hat{R}$, with fixed initial extended clusters $\mathbf{x}_0$, $\tilde{\mathbf{x}}_0$ and $\hat{\mathbf{x}}_0$. Assume the following:
    \begin{enumerate}[1)]
        \item There are markers $\kappa$ and $\kappa^\prime$ that make $(\mathcal{C},\tilde{\mathcal{C}},\kappa)$ and $(\mathcal{C},\hat{\mathcal{C}},\kappa^\prime)$ related triples, as well as a pair of complementary birational quasi-isomorphisms $\mathcal{Q}:(\mathcal{C},R) \rightarrow (\tilde{\mathcal{C}},\tilde{R})$ and $\mathcal{Q}^\prime:(\mathcal{C},R)\rightarrow(\hat{\mathcal{C}},\hat{R})$; \label{i:fbird_1}
        \item The initial extended clusters $\tilde{\mathbf{x}}_0$ and $\hat{\mathbf{x}}_0$ are i-admissible;
        \item The initial extended cluster $\mathbf{x}_0$ is regular.
    \end{enumerate}
    Then $\mathbf{x}_0$ is i-admissible, as well as $\bar{\mathcal{A}}_{\mathbb{K}}(\hat{\mathcal{C}}) \subseteq \hat{R}$, $\bar{\mathcal{A}}_{\mathbb{K}}(\tilde{\mathcal{C}}) \subseteq \tilde{R}$ and $\bar{\mathcal{A}}_{\mathbb{K}}(\mathcal{C}) \subseteq R$. 
\end{proposition}

\begin{remark}\label{r:factorial}
    The coordinate ring $\mathbb{C}[G]$ of a simply connected simple complex algebraic group $G$ is a UFD such that $(\mathbb{C}[G])^{\times} = \mathbb{C}^\times$ (see, for instance, \cite{popov}). Thus Propositions~\ref{p:fbirat_single} and \ref{p:birat_double_ufd} cover the needs of the Gekhtman--Shapiro--Vainshtein program~\cite{rho, conj} on the construction of multiple generalized cluster structures in $\mathbb{C}[G]$ for  a given group $G$. At the cost of conciseness, both propositions can be extended in various ways. For instance, one can require that $R^{\times}$ and $\tilde{R}^{\times}$ are groups of Laurent monomials in the invertible nonmarked frozen variables with coefficients in $\mathbb{K}^\times$ (thus covering the case of $\mathbb{C}[\GL_n]$). 
\end{remark}

\begin{remark}
    In the definition of an i-admissible set of variables, we include the requirement that the frozen variables are  irreducible. This is only a matter of convenience: we tailor the statement towards our applications (see Remark~\ref{r:prime}). If the frozen variables in $\tilde{R}$ are not irreducible, they may not be irreducible in $R$, but the statements of the above results are still valid.
\end{remark}

The next propositions concern the reverse inclusion $\bar{\mathcal{A}}_{\mathbb{K}}(\mathcal{C}) \supseteq R$ and are based on the result on upper bounds (see Propositions~\ref{p:upper_bound}) in case there are birational quasi-isomorphisms. These results were implicitly present in the proof of \cite[Theorem 3.11]{plethora}. For the sake of completeness, we record them in the setup of the paper.

Recall that given an extended cluster $\mathbf{x}$, we denote by $\mathcal{L}(\mathbf{x})$ the Laurent ring
\begin{equation}
    \mathcal{L}(\mathbf{x}) := \mathbb{Z}[x_1^{\pm 1}, x_2^{\pm 1}, \ldots, x_N^{\pm 1}, x_{N+1}, \ldots, x_{N+M}],
\end{equation}
where the localization is taken at the cluster variables. We denote by $\mathbf{x}_i^\prime$ the mutation of the initial extended cluster $\mathbf{x}_0$ in direction $i \in [1,N]$.

\begin{proposition}\label{p:upper_cont1}
    Let $(\mathcal{C},R)$ and $(\tilde{\mathcal{C}},\tilde{R})$ be cluster structures of geometric type in normal Noetherian domains $R$ and $\tilde{R}$ with fixed initial extended seeds $(\mathbf{x}_0,B_0)$ and $(\tilde{\mathbf{x}}_0,\tilde{B}_0)$. Assume the following:
    \begin{enumerate}[1)]
        \item There is a marker $\kappa$ that makes $(\mathcal{C},\tilde{\mathcal{C}},\kappa)$ a related triple, along with a birational quasi-isomorphism $\mathcal{Q}:(\mathcal{C},R)\rightarrow(\tilde{\mathcal{C}},\tilde{R})$;
        \item The matrices $B_0$ and $\tilde{B}_0$ are of full rank;
        \item The reverse inclusion $\bar{\mathcal{A}} (\tilde{\mathcal{C}})\supseteq \tilde{R}$ holds;
        \item For each $i \in \mathcal{I}(\kappa)$, $\mathcal{L}(\mathbf{x}_i^\prime)\supseteq R$.
    \end{enumerate}
    Then $\bar{\mathcal{A}}(\mathcal{C})\supseteq R$.
\end{proposition}

\begin{remark}
    If $R$ is in addition a $\mathbb{K}$-algebra, one has to substitute $\mathcal{L}(\mathbf{x}_i^\prime)$ and $\bar{\mathcal{A}}(\mathcal{C})$ with $\mathcal{L}_{\mathbb{K}}(\mathbf{x}_i^\prime):=\mathcal{L}(\mathbf{x}_i^\prime)\otimes_{\mathbb{Z}} \mathbb{K}$ and $\bar{\mathcal{A}}_{\mathbb{K}}(\mathcal{C}):=\bar{\mathcal{A}}(\mathcal{C})\otimes_{\mathbb{Z}} \mathbb{K}$ (otherwise, the conditions of Proposition~\ref{p:upper_cont1} can never be satisfied).
\end{remark}

\begin{proposition}\label{p:upper_cont}
        Let $R$, $\tilde{R}$ and $\hat{R}$ be normal Noetherian domains, $(\mathcal{C},R)$, $(\tilde{\mathcal{C}},\tilde{R})$ and $(\hat{\mathcal{C}},\hat{R})$ be cluster structures of geometric type in $R$, $\tilde{R}$ and $\hat{R}$, with fixed initial extended seeds $(\mathbf{x}_0,B_0)$, $(\tilde{\mathbf{x}}_0,\tilde{B}_0)$ and $(\hat{\mathbf{x}}_0,\hat{B}_0)$.  Assume the following:
        \begin{enumerate}[1)]
            \item There are markers $\kappa$ and $\kappa^\prime$ that make $(\mathcal{C},\tilde{\mathcal{C}},\kappa)$ and $(\mathcal{C},\hat{\mathcal{C}},\kappa^\prime)$ related triples, as well as a pair of complementary birational quasi-isomorphisms $\mathcal{Q}:(\mathcal{C},R) \rightarrow (\tilde{\mathcal{C}},\tilde{R})$ and $\mathcal{Q}^\prime:(\mathcal{C},R)\rightarrow(\hat{\mathcal{C}},\hat{R})$;
            \item The matrices $\tilde{B}_0$, $\hat{B}_0$ and $B_0$ are of full rank;
            \item The inclusions $\bar{\mathcal{A}}(\tilde{\mathcal{C}}) \supseteq \tilde{R}$ and $\bar{\mathcal{A}}(\hat{\mathcal{C}}) \supseteq \hat{R}$ hold.
        \end{enumerate}
        Then $\bar{\mathcal{A}}(\mathcal{C}) \supseteq R$.
\end{proposition}

\begin{remark}\label{r:gcs_extend}
    None of the proofs of the above results uses a specific form of the mutation relation. The proofs rely on the following facts: 1) Starfish lemma (implicitly on the Laurent phenomenon); 2) If $\mathbf{x}_0$ is regular, then $x^\prime \in R[x^{\pm 1}]$ for every cluster variable $x\in\mathbf{x}_0$; 3) The equality $\bar{\mathcal{A}}(\mathcal{C}) = \bar{\mathcal{A}}(\mathcal{N})$ for a star-shaped nerve $\mathcal{N}$ in~$\mathcal{C}$. Thus the above results hold for more general cluster structures in which the mentioned facts are true. In particular, if Remark~\ref{r:extension} is taken into account, the above results hold for generalized cluster structures as defined in~\cite{earlier,double}. 
\end{remark}

%% file: birat_single.tex
The setup for this section is as follows. Let $R$ and $\tilde{R}$ be normal Noetherian domains. Assume that $(\mathcal{C},R)$ and $(\tilde{\mathcal{C}},\tilde{R})$ are cluster structures of geometric type in $R$ and $\tilde{R}$ with fixed initial extended seeds $(\mathbf{x}_0,B_0)$ and $(\tilde{\mathbf{x}}_0,\tilde{B}_0)$, and such that $(\mathcal{C},\tilde{\mathcal{C}},\kappa)$ is a related triple. We also assume there is a birational quasi-isomorphism $\mathcal{Q}:(\mathcal{C},R)\rightarrow(\tilde{\mathcal{C}},\tilde{R})$. The objective of this section is to prove results that involve the use of a single birational quasi-isomorphism; that is, Proposition~\ref{p:birat_single}, Proposition~\ref{p:fbirat_single} and Proposition~\ref{p:upper_cont1}, along with less important results. We note that, according to Definition~\ref{d:birat}, the marked variables $\mathbf{x}_0(\kappa)$ and $\tilde{\mathbf{x}}_0(\kappa)$ are assumed to be regular.

\subsection{Preliminary lemmas}\label{s:prelim}
In this subsection, we establish a few preliminary lemmas. 
\begin{lemma}\label{l:loc_inters}
    Let $A$ and $B$ be multiplicative subsets of $R$ that do not contain $0$. Assume that each element of $A$ is coprime with each element of $B$. Then $R = A^{-1}R \cap B^{-1}R$.
\end{lemma}
\begin{proof}
    Recall that for normal Noetherian domains, $R = \bigcap_{\htt \mathfrak{p} = 1} R_{\mathfrak{p}}$ (the intersection is over prime ideals of height $1$) and that  $\Spec(A^{-1}R)$ is homeomorphic to $\{\mathfrak{p}\in\Spec(R) \ | \ \mathfrak{p}\cap A = \emptyset\}$. Since $A^{-1}R$ is also a normal Noetherian domain, we can write
    \begin{equation}
    A^{-1}R = \bigcap_{\substack{\mathfrak{p}\cap A =\emptyset, \\ \htt \mathfrak{p}=1}} R_{\mathfrak{p}},
    \end{equation}
    and likewise for $B$. Since all elements in $A$ are coprime with all elements in $B$, for any $\mathfrak{p}\in \Spec(R)$ of height $1$,  $\mathfrak{p} \cap A = \emptyset$ or $\mathfrak{p}\cap B = \emptyset$. This implies that
    \begin{equation}
        R = \bigcap_{\htt \mathfrak{p} = 1}R_{\mathfrak{p}} = \left(\bigcap_{\substack{\mathfrak{p}\cap A = \emptyset, \\ \htt \mathfrak{p}=1}}R_{\mathfrak{p}}\right) \cap \left(\bigcap_{\substack{\mathfrak{p}\cap B = \emptyset, \\ \htt \mathfrak{p}=1}}R_{\mathfrak{p}}\right) = A^{-1}R \cap B^{-1}R.
    \end{equation}
    Thus the statement holds.
\end{proof}
 For an extended marked cluster $\mathbf{x}$ and a regular variable $x \in \mathbf{x}$, let $\mathcal{P}(x)$ be the set of prime divisors\footnote{Let us recall that given an ideal $I$ of a commutative ring $R$, a prime ideal $\mathfrak{p}$ is a \emph{prime divisor} of $I$ if $I \subseteq \mathfrak{p}$. It is a well-known result in commutative algebra that every principal ideal has a finite number of prime divisors, and all of such prime divisors have height $1$.} of $(x)$ in $R$, and likewise (if the related variable $\tilde{x}$ is regular), $\mathcal{P}(\tilde{x})$ be the set of prime divisors of $(\tilde{x})$ in $\tilde{R}$. Note that since $R$ is a normal Noetherian domain, $\mathcal{P}(x)$ consists of minimal prime ideals of $(x)$, and each of them is of height $1$; likewise for $(\tilde{x})$. Set \begin{align}&\mathcal{P}_{\kappa}(x):=\{\mathfrak{p}\in \mathcal{P}(x) \ | \ \mathfrak{p}\cap \mathbf{x}_0(\kappa)=\emptyset\},\\
 &\mathcal{P}_{\kappa}(\tilde{x}):=\{\tilde{\mathfrak{p}} \in \mathcal{P}(\tilde{x}) \ | \ \tilde{\mathfrak{p}}\cap \tilde{\mathbf{x}}_0(\kappa)=\emptyset\}
 \end{align}
 to be the prime divisors of $(x)$ (respectively of $(\tilde{x})$) that do not contain the marked variables.

\begin{lemma}\label{l:minprim}
    Let $x$ be a nonmarked cluster or frozen variable from an extended marked cluster in $\mathcal{C}$. Assume that both $x$ and its related variable $\tilde{x}$ are regular. Then the birational quasi-isomorphism $\mathcal{Q}$ induces a bijection between $\mathcal{P}_{\kappa}(x)$ and $\mathcal{P}_{\kappa}(\tilde{x})$. 
\end{lemma}
\begin{proof}
As an isomorphism $\mathcal{Q}:R[\mathbf{x}_0(\kappa)^{\pm 1}] \xrightarrow{\sim} \tilde{R}[\tilde{\mathbf{x}}_0(\kappa)^{\pm 1}]$, $\mathcal{Q}$ sends $x$ to an associate of $\tilde{x}$; therefore, $\Spec \mathcal{Q}$ induces a bijection between the prime divisors of $(x)$ in $R[\mathbf{x}_0(\kappa)^{\pm 1}]$ and the prime divisors of $(\tilde{x})$ in $\tilde{R}[\tilde{\mathbf{x}}_0(\kappa)^{\pm 1}]$; since the prime divisors contract to the elements of $\mathcal{P}_{\kappa}(x)$ and $\mathcal{P}_{\kappa}(\tilde{x})$, we obtain a bijection $\mathcal{P}_{\kappa}(x) \cong \mathcal{P}_{\kappa}(\tilde{x})$.
\end{proof}

\subsection{Transferring cluster and algebraic properties}

In this subsection, we show how to derive coprimality and regularity of nonmarked variables in $\mathcal{C}$ given their regularity and coprimality in $\tilde{\mathcal{C}}$.

\begin{proposition}\label{p:varcoprim}
    Let $x$ and $z$ be any two nonmarked cluster or frozen variables in $\mathcal{C}$ that belong to (possibly different) extended marked clusters. Assume the following conditions:
    \begin{enumerate}[1)]
        \item The variables $x$, $z$, and the related variables $\tilde{x}$ and $\tilde{z}$ are regular;
        \item The variables $\tilde{x}$ and $\tilde{z}$ are coprime;
        \item The variable $x$ is coprime with the marked variables.
    \end{enumerate}
    Then the variables $x$ and $z$ are coprime.
\end{proposition}
\begin{proof}
    Since $x$ is coprime with the marked variables, $\mathcal{P}(x) = \mathcal{P}_{\kappa}(x)$ and $\mathcal{P}(x)\cap (\mathcal{P}(z)\setminus \mathcal{P}_{\kappa}(z)) = \emptyset$.
    By Lemma~\ref{l:minprim}, $\mathcal{Q}$ induces bijections $\mathcal{P}(x)\cong \mathcal{P}_{\kappa}(\tilde{x})$ and $\mathcal{P}_{\kappa}(z)\cong \mathcal{P}_{\kappa}(\tilde{z})$. Since $\tilde{x}$ and $\tilde{z}$ are coprime, $\mathcal{P}_{\kappa}(\tilde{x})\cap \mathcal{P}_{\kappa}(\tilde{z}) = \emptyset$, hence $\mathcal{P}(x)\cap \mathcal{P}_{\kappa}(z) = \emptyset$. Thus $x$ and $z$ are coprime in $R$.
\end{proof}

\begin{proposition}\label{p:copmut}
    Let $x$ be a nonmarked cluster variable from an extended marked cluster of $\mathcal{C}$. Assume the following:
    \begin{enumerate}[1)]
    \item Both $\tilde{x}$ and its mutation $\tilde{x}^\prime$ are regular and coprime;
    \item The variable $x$ is regular and coprime with the marked variables;
    \item The exchange polynomial $xx^\prime$ is regular.
    \end{enumerate}
    Then $x^\prime$ is regular and is coprime with $x$.
\end{proposition}
\begin{proof}
    Indeed, since $\tilde{x}^\prime$ is regular in $\tilde{R}$, $x^\prime \in R[x^{\pm 1}]\cap R[\mathbf{x}_0(\kappa)^{\pm 1}]$. Since $x$ is coprime with the marked variables, by Lemma~\ref{l:loc_inters}, $x^\prime \in R$. The coprimality follows from Proposition~\ref{p:varcoprim} via $z:=x^\prime$.
\end{proof}

\subsection{Prime elements}
We continue in the setup of the section. Here, we show that if $\tilde{x}$ is a regular cluster or frozen variable that is also a prime element in $\tilde{R}$, then, under certain conditions, the related variable $x$ is also a prime element in $R$.

\begin{lemma}\label{l:prime_gap}
    Let $A$ be a multiplicatively closed subset of $R \setminus \{0\}$, and let $x \in R$ be an element coprime with $A$. If $(x)\cdot A^{-1}R$ is prime in $A^{-1}R$, then $(x)$ is prime in $R$.
\end{lemma}
\begin{proof}
Since the extension $(x)\cdot A^{-1}R$ is prime, the contraction $\mathfrak{p}:=((x)\cdot A^{-1}R)\cap R$ is prime and contains $(x)$. Since $x$ is coprime with $A$, and since $R$ is a normal Noetherian domain, $\mathfrak{p}$ is the only prime divisor of $x$. This implies that $(x)$ admits a primary decomposition $(x) = \bigcap_{i=1}^n \mathfrak{q}_i$ where each $\mathfrak{q}_i$ is $\mathfrak{p}$-primary. Since a finite intersection of $\mathfrak{p}$-primary ideals is $\mathfrak{p}$-primary, we conclude that $(x)$ is a $\mathfrak{p}$-primary ideal. Now recall that if $\mathfrak{q}$ is a $\mathfrak{p}$-primary ideal such that $\mathfrak{q}\cap A = \emptyset$, then $(\mathfrak{q}\cdot A^{-1}R)\cap R = \mathfrak{q}$; therefore, $(x) = \mathfrak{p}$, and thus $x$ is prime in $R$.
\end{proof}


\begin{proposition}\label{p:prime_el}
    Let $x$ be a regular nonmarked variable from an extended marked cluster in $\mathcal{C}$. Assume that $\tilde{x}$ is a prime element in $\tilde{R}$, and assume that both $x$ and $\tilde{x}$ are coprime with the marked variables. Then $x$ is prime in $R$.
\end{proposition}
\begin{proof}
    Since $\tilde{x}$ is coprime with the marked variables, $\tilde{x}$ is prime in $\tilde{R}[\tilde{\mathbf{x}}_0(\kappa)^{\pm 1}]$, and since $\mathcal{Q}$ sends $\tilde{x}$ to an associate of $x$ in $R[\mathbf{x}_0(\kappa)^{\pm 1}]$, $x$ is prime in the latter ring. By Lemma~\ref{l:prime_gap}, $x$ is prime in $R$.
\end{proof}

\begin{remark}\label{r:prime}
    When a cluster structure is compatible with a Poisson bracket (in the sense of~\cite{dasbuch}), frozen variables are observed to have the following property: their zero loci foliate into unions of symplectic leaves of the ambient Poisson variety. To prove the latter, one may use an algebraic criterion from~\cite[Proposition 2.3]{yakimov_locus} or~\cite[Remark 2.4]{yakimov_det}, which in particular requires showing that a given frozen variable is a prime element in the coordinate ring of the given variety. Proposition~\ref{p:prime_el} provides a way of transferring this property via a birational quasi-isomorphism.  For an example of the use of Proposition~\ref{p:prime_el} in this context, one may refer to our latest work~\cite{multdual}.
\end{remark}

\subsection{Proof of Proposition~\ref{p:birat_single}}

Recall that a set of initial variables is admissible if it satisfies the conditions of the Starfish lemma (see Definition~\ref{d:admissible}). In this subsection, we prove Proposition~\ref{p:birat_single}, which we restate for convenience.

\begin{proposition*}
     Let $(\mathcal{C},R)$ and $(\tilde{\mathcal{C}},\tilde{R})$ be cluster structures of geometric type in normal Noetherian domains $R$ and $\tilde{R}$ with fixed initial extended clusters $\mathbf{x}_0$ and $\tilde{\mathbf{x}}_0$. Assume the following:
     \begin{enumerate}[1)]
         \item There is a marker $\kappa$ that makes $(\mathcal{C},\tilde{\mathcal{C}},\kappa)$ a related triple, along with a birational quasi-isomorphism $\mathcal{Q}:(\mathcal{C},R)\rightarrow(\tilde{\mathcal{C}},\tilde{R})$;\label{ii:birsing1}
         \item The initial extended cluster $\tilde{\mathbf{x}}_0$ is admissible;\label{ii:birsing2}
         \item The set of marked variables $\mathbf{x}_0(\kappa)$ is admissible;\label{ii:birsing3}
        \item The initial extended cluster $\mathbf{x}_0$ is regular;\label{ii:birsing4}
         \item Each nonmarked cluster variable in $\mathbf{x}_0$ is coprime with each marked variable. \label{ii:birsing5}
    \end{enumerate}
    Then $\mathbf{x}_0$ is admissible, $\bar{\mathcal{A}}(\tilde{\mathcal{C}}) \subseteq \tilde{R}$ and $\bar{\mathcal{A}}(\mathcal{C}) \subseteq R$.
\end{proposition*}
\begin{proof}
It follows from the Starfish lemma and Condition~\ref{ii:birsing2} that $\bar{\mathcal{A}}(\tilde{\mathcal{C}}) \subseteq \tilde{R}$. To deduce $\bar{\mathcal{A}}(\mathcal{C}) \subseteq R$, we need to show that $\mathbf{x}_0$ is admissible. By Proposition~\ref{p:varcoprim}, we see that the initial cluster variables in $\mathbf{x}_0$ are pairwise coprime; by Proposition~\ref{p:copmut}, we see that for each initial cluster variable $x \in \mathbf{x}_0$, $x^\prime$ is regular and coprime with $x$. Thus $\bar{\mathcal{A}}(\mathcal{C}) \subseteq R$.
\end{proof}

\subsection{Proof of Proposition~\ref{p:fbirat_single}}

Recall that an initial set of variables is i-admissible if the variables are regular, irreducible, and (for cluster variables) coprime with their mutations (see Definition~\ref{d:iadmissible}). In this subsection, we prove Proposition~\ref{p:fbirat_single}, which we restate for convenience.

\begin{proposition*}
   Let $(\mathcal{C},R)$ and $(\tilde{\mathcal{C}},\tilde{R})$ be cluster structures of geometric type in factorial Noetherian $\mathbb{K}$-algebras with no zero divisors and such that $R^{\times} = \tilde{R}^\times = \mathbb{K}^{\times}$. Let $\mathbf{x}_0$ and $\tilde{\mathbf{x}}_0$ be fixed initial extended clusters. Assume the following:
     \begin{enumerate}[1)]
         \item There is a marker $\kappa$ that makes $(\mathcal{C},\tilde{\mathcal{C}},\kappa)$ a related triple, along with a birational quasi-isomorphism $\mathcal{Q}:(\mathcal{C},R)\rightarrow(\tilde{\mathcal{C}},\tilde{R})$;\label{ii:fbirsing1}
         \item The initial extended cluster $\tilde{\mathbf{x}}_0$ is i-admissible;\label{ii:fbirsing2}
         \item The set of marked variables $\mathbf{x}_0(\kappa)$ is i-admissible;\label{ii:fbirsing3}
        \item The initial extended cluster $\mathbf{x}_0$ is regular;\label{ii:fbirsing4}
        \item Each nonmarked cluster variable in $\mathbf{x}_0$ is coprime with each marked variable.\label{ii:fbirsing5}
    \end{enumerate}
    Then $\mathbf{x}_0$ is i-admissible, as well as $\bar{\mathcal{A}}_{\mathbb{K}}(\tilde{\mathcal{C}}) \subseteq \tilde{R}$ and $\bar{\mathcal{A}}_{\mathbb{K}}(\mathcal{C}) \subseteq R$.
\end{proposition*}
\begin{proof}
    Since the initial variables $\tilde{\mathbf{x}}_0$ are algebraically independent over $\mathbb{Q}$, they are algebraically independent over $\mathbb{K}$, and since $\tilde{R}^\times = \mathbb{K}^\times$, the irreducibility of the initial variables implies their coprimality; in other words, $\tilde{\mathbf{x}}_0$ is admissible. Likewise, $\mathbf{x}_0(\kappa)$ is admissible. By Proposition~\ref{p:prime_el}, we see that $\mathbf{x}_0$ is i-admissible, and since the variables in $\mathbf{x}_0$ are algebraically independent and $R^\times = \mathbb{K}^\times$, $\mathbf{x}_0$ is admissible. Now the statement follows from Proposition~\ref{p:birat_single}.
\end{proof}

\subsection{Proof of Proposition~\ref{p:upper_cont1}}

We restate Proposition~\ref{p:upper_cont1} for convenience.

\begin{proposition*}
    Let $(\mathcal{C},R)$ and $(\tilde{\mathcal{C}},\tilde{R})$ be cluster structures of geometric type in normal Noetherian domains $R$ and $\tilde{R}$ with fixed initial extended seeds $(\mathbf{x}_0,B_0)$ and $(\tilde{\mathbf{x}}_0,\tilde{B}_0)$. Assume the following:
    \begin{enumerate}[1)]
        \item There is a marker $\kappa$ that makes $(\mathcal{C},\tilde{\mathcal{C}},\kappa)$ a related triple, along with a birational quasi-isomorphism $\mathcal{Q}:(\mathcal{C},R)\rightarrow(\tilde{\mathcal{C}},\tilde{R})$;
        \item The matrices $B_0$ and $\tilde{B}_0$ are of full rank;
        \item The reverse inclusion $\bar{\mathcal{A}} (\tilde{\mathcal{C}})\supseteq \tilde{R}$ holds;
        \item For each $i \in \mathcal{I}(\kappa)$, $\mathcal{L}(\mathbf{x}_i^\prime)\supseteq R$.
    \end{enumerate}
    Then $\bar{\mathcal{A}}(\mathcal{C})\supseteq R$.
\end{proposition*}
\begin{proof}
    Let $\tilde{\mathcal{N}}$ and $\mathcal{N}$ be star-shaped nerves centered at $\tilde{\mathbf{x}}_0$ and  $\mathbf{x}_0$, respectively. By Proposition~\ref{p:upper_bound}, $\bar{\mathcal{A}}(\mathcal{C}) = \bar{\mathcal{A}}(\mathcal{N})$ and $\bar{\mathcal{A}}(\tilde{\mathcal{C}}) = \bar{\mathcal{A}}(\tilde{\mathcal{N}})$. For each mutation $\mathbf{x}_i^\prime$ of $\mathbf{x}_0$ in direction $i \in [1,N]\setminus\mathcal{I}(\kappa)$, there is an isomorphism of rings $\mathcal{Q}:\mathcal{L}(\mathbf{x}_i^\prime)\xrightarrow{\sim} \mathcal{L}(\tilde{\mathbf{x}}_{\kappa(i)}^\prime)[\tilde{\mathbf{x}}_0(\kappa)^{\pm 1}]$; therefore,
        \begin{equation}\label{eq:barincl}
            R\subseteq \mathcal{Q}^{-1}(\tilde{R}[\tilde{\mathbf{x}}_0(\kappa)^{\pm 1}]) \subseteq \bigcap_{\tilde{\mathbf{x}}\in\tilde{\mathcal{N}}}\mathcal{Q}^{-1}\left(\mathcal{L}(\tilde{\mathbf{x}})[\tilde{\mathbf{x}}_0(\kappa)^{\pm 1}] \right) = \bigcap_{i\in [1,N]\setminus\mathcal{I}(\kappa)}\mathcal{L}(\mathbf{x}_i^\prime),
        \end{equation}
        and since $R\subseteq \mathcal{L}(\mathbf{x}_i^\prime)$ for $i\in\mathcal{I}(\kappa)$, we conclude that $R \subseteq \bar{\mathcal{A}}(\mathcal{C})$.
\end{proof}

%% file: birat_double.tex
In this section, we prove the results that involve the use of a pair of complementary birational quasi-isomorphisms; that is, Proposition~\ref{p:birat_double}, Proposition~\ref{p:fbirat_single} and Proposition~\ref{p:upper_cont}.

The setup for the section is as follows. Let $(\mathcal{C},R)$, $(\tilde{\mathcal{C}},\tilde{R})$ and $(\hat{\mathcal{C}},\hat{R})$ be cluster structures of geometric type in normal Noetherian domains $R$, $\tilde{R}$ and $\hat{R}$ with fixed initial extended seeds $(\mathbf{x}_0,B_0)$,  $(\tilde{\mathbf{x}}_0,\tilde{B}_0)$ and $(\hat{\mathbf{x}}_0,\hat{B}_0)$, such that $(\mathcal{C},\tilde{\mathcal{C}},\kappa)$ and $(\mathcal{C},\hat{\mathcal{C}},\kappa^\prime)$ are related triples for some markers $\kappa$ and $\kappa^\prime$. We assume that there exists a pair of complementary birational quasi-isomorphisms $\mathcal{Q}:(\mathcal{C},R)\rightarrow(\tilde{\mathcal{C}},\tilde{R})$ and $\mathcal{Q}^\prime:(\mathcal{C},R) \rightarrow (\hat{\mathcal{C}},\hat{R})$; that is, we assume $\mathcal{I}(\kappa)\cap\mathcal{I}(\kappa^\prime) = \emptyset$. Similarly to the case of $\tilde{\mathcal{C}}$, all the data associated with $\hat{\mathcal{C}}$ is decorated with a hat. Along with $\tilde{\mathbf{x}}_0(\kappa)$ and $\hat{\mathbf{x}}_0(\kappa^\prime)$ defined in~\eqref{eq:tildx0kap}, we also set
\begin{align}
    &\tilde{\mathcal{I}}(\kappa^\prime):=\kappa(\mathcal{I}(\kappa^\prime)),& &\hat{\mathcal{I}}(\kappa):=\kappa^\prime(\mathcal{I}(\kappa)),\\
    &\tilde{\mathbf{x}}_0(\kappa^\prime):= \{x_i \in \tilde{\mathbf{x}}_0 \ | \ i \in \tilde{\mathcal{I}}(\kappa^\prime)\},& &\hat{\mathbf{x}}_0(\kappa):= \{x_i \in \hat{\mathbf{x}}_0 \ | \ i \in \hat{\mathcal{I}}(\kappa)\}.
\end{align}
In other words, $\tilde{\mathbf{x}}_0(\kappa^\prime)$ are the variables in $\tilde{\mathbf{x}}_0$ that correspond to the marked variables of the related triple $(\mathcal{C},\hat{\mathcal{C}},\kappa^\prime)$, and likewise for $\hat{\mathbf{x}}_0(\kappa)$. 

\begin{definition}
    Given a pair of related triples $(\mathcal{C},\tilde{\mathcal{C}},\kappa)$ and $(\mathcal{C},\tilde{\mathcal{C}},\kappa^\prime)$ as above, a variable in
    $\mathcal{C}$, $\tilde{\mathcal{C}}$ or $\hat{\mathcal{C}}$ is called \emph{marked} if it is marked with respect to one of the triples. An extended cluster is \emph{marked} if it is marked relative both related triples.   
\end{definition}

\subsection{Transferring cluster and algebraic properties}

We begin with an auxiliary result.

\begin{proposition}\label{p:cbir_vars}
    Let $x$ be a nonmarked cluster or frozen variable from an extended marked cluster in $\mathcal{C}$. Assume the following:
    \begin{enumerate}[1)]
        \item The variable $\tilde{x}$ is 
         regular and coprime with $\tilde{\mathbf{x}}_0(\kappa^\prime)$;\label{p:cbir1}
        \item The variable $\hat{x}$ is regular and coprime with $\hat{\mathbf{x}}_0(\kappa)$;\label{p:cbir2}
        \item Every $x_i \in \mathbf{x}_0(\kappa)$ is coprime with every $x_j \in \mathbf{x}_0(\kappa^\prime)$.\label{p:cbir3}
    \end{enumerate}
    Then $x$ is regular and coprime with all the marked variables. 
\end{proposition}
\begin{proof}
Since $\mathcal{Q}$ induces an isomorphism of localizations $\mathcal{Q} : R[\mathbf{x}_0(\kappa)^{\pm 1}] \xrightarrow{\sim} \tilde{R}[\tilde{\mathbf{x}}_0(\kappa)^{\pm 1}]$ and since $\tilde{x} \in \tilde{R}$, we see that $x \in R[\mathbf{x}_0(\kappa)^{\pm 1}]$; likewise, since $\hat{x}$ is regular, we see that $x \in R[\mathbf{x}_0(\kappa^\prime)^{\pm 1}]$. From Lemma~\ref{l:loc_inters} and Condition~\ref{p:cbir3}, we conclude that $R = R[\mathbf{x}_0(\kappa)^{\pm 1}] \cap R[\mathbf{x}_0(\kappa^\prime)^{\pm 1}]$, and therefore, $x \in R$. To deduce coprimality between $x$ and a marked variable $x_j \in \mathbf{x}_0(\kappa^\prime)$, we apply Proposition~\ref{p:varcoprim} relative $(\mc,\tilde{\mc},\kappa)$; specifically, the variable $z$ from Proposition~\ref{p:varcoprim} is the variable $x$, and the variable $x$ from Proposition~\ref{p:varcoprim} is the variable $x_j$ (for which the coprimality with $\mathbf{x}_0(\kappa)$ is assumed); hence, $x$ is coprime with all the variables in $\mathbf{x}_0(\kappa^\prime)$. Similarly, the coprimality between $x$ and the marked variables in $\mathbf{x}_0(\kappa)$ follows from Proposition~\ref{p:varcoprim} applied to $(\mc,\hat{\mc},\kappa^\prime)$. \qedhere
\end{proof}

\subsection{Proof of Proposition~\ref{p:birat_double}}
We restate Proposition~\ref{p:birat_double} for convenience.

\begin{proposition*}
    Let $R$, $\tilde{R}$ and $\hat{R}$ be normal Noetherian domains, $(\mathcal{C},R)$, $(\tilde{\mathcal{C}},\tilde{R})$ and $(\hat{\mathcal{C}},\hat{R})$ be cluster structures of geometric type in $R$, $\tilde{R}$ and $\hat{R}$, with fixed initial extended clusters $\mathbf{x}_0$, $\tilde{\mathbf{x}}_0$ and $\hat{\mathbf{x}}_0$. Assume the following:
    \begin{enumerate}[1)]
        \item There are markers $\kappa$ and $\kappa^\prime$ that make $(\mathcal{C},\tilde{\mathcal{C}},\kappa)$ and $(\mathcal{C},\hat{\mathcal{C}},\kappa^\prime)$ related triples, as well as a pair of complementary birational quasi-isomorphisms $\mathcal{Q}:(\mathcal{C},R) \rightarrow (\tilde{\mathcal{C}},\tilde{R})$ and $\mathcal{Q}^\prime:(\mathcal{C},R)\rightarrow(\hat{\mathcal{C}},\hat{R})$; \label{ii:bird_1}
        \item The initial extended clusters $\tilde{\mathbf{x}}_0$ and $\hat{\mathbf{x}}_0$ are admissible; \label{ii:bird_2}
        \item Every $x_i \in \mathbf{x}_0(\kappa)$ is coprime with every $x_j \in \mathbf{x}_0(\kappa^\prime)$; \label{ii:bird_3}
        \item The initial extended cluster $\mathbf{x}_0$ is regular. \label{ii:bird_4}
    \end{enumerate}
    Then $\mathbf{x}_0$ is admissible, as well as $\bar{\mathcal{A}}(\hat{\mathcal{C}}) \subseteq \hat{R}$, $\bar{\mathcal{A}}(\tilde{\mathcal{C}}) \subseteq \tilde{R}$ and $\bar{\mathcal{A}}(\mathcal{C}) \subseteq R$.
\end{proposition*}
\begin{proof}
As in the proof of Proposition~\ref{p:birat_single}, Condition~\ref{ii:bird_2} is enough to conclude $\bar{\mathcal{A}}(\hat{\mathcal{C}}) \subseteq \hat{R}$ and $\bar{\mathcal{A}}(\tilde{\mathcal{C}}) \subseteq \tilde{R}$.  To claim $\bar{\mathcal{A}}(\mathcal{C}) \subseteq R$, we need to show that $\mathbf{x}_0$ is admissible. By Proposition~\ref{p:cbir_vars}, we see that each nonmarked initial variable $x \in \mathbf{x}_0$ is coprime with all the marked variables, and if $x$ is in addition a cluster variable, by Proposition~\ref{p:copmut}, $x^\prime$ is regular and coprime with $x$. By Proposition~\ref{p:varcoprim} applied to $x,z \in \mathbf{x}_0 \setminus \mathbf{x}_0(\kappa)$ relative $(\mathcal{C},\tilde{\mathcal{C}},\kappa)$, we see that $x$ and $z$ are coprime (and likewise, if $x,z \in \mathbf{x}\setminus \mathbf{x}_0(\kappa^\prime)$, apply the proposition relative $(\mathcal{C},\hat{\mathcal{C}},\kappa^\prime)$). By Proposition~\ref{p:copmut}, the mutation $x^\prime$ of a marked variable $x$ is regular and coprime with $x$. Therefore, $\mathbf{x}_0$ is admissible, and thus $\bar{\mathcal{A}}(\mathcal{C}) \subseteq R$.
\end{proof}

\subsection{Proof of Proposition~\ref{p:birat_double_ufd}}
For this subsection, let us introduce additional notation:
\begin{equation}
\hat{\mathbf{x}}_0(\kappa,\kappa^\prime):= \hat{\mathbf{x}}_0(\kappa) \cup \hat{\mathbf{x}}_0(\kappa^\prime), \ \ \tilde{\mathbf{x}}_0(\kappa,\kappa^\prime):= \tilde{\mathbf{x}}_0(\kappa) \cup \tilde{\mathbf{x}}_0(\kappa^\prime), \ \  \mathbf{x}_0(\kappa,\kappa^\prime):= \mathbf{x}_0(\kappa) \cup \mathbf{x}_0(\kappa^\prime).
\end{equation}

\begin{lemma}\label{l:cv_fact_prim}
    Assume that $R$, $\tilde{R}$ and $\hat{R}$ are factorial Noetherian $\mathbb{K}$-algebras with no zero divisors and such that $R^{\times} = \tilde{R}^\times = \hat{R}^{\times} = \mathbb{K}^{\times}$. Assume that all the variables in $\tilde{\mathbf{x}}_0(\kappa,\kappa^\prime)$ and $\hat{\mathbf{x}}_0(\kappa,\kappa^\prime)$ are prime elements in $\tilde{R}$ and $\hat{R}$, respectively. Then $\mathbf{x}_0(\kappa,\kappa^\prime)$ are  prime elements in $R$.
\end{lemma}
\begin{proof}
As in the proof of Proposition~\ref{p:fbirat_single}, since the variables are algebraically independent over $\mathbb{K}$, we see that the marked variables in $\tilde{R}$ and $\hat{R}$ are pairwise coprime. By Proposition~\ref{p:prime_el}, it is enough to show that each $x_i \in \mathbf{x}_0(\kappa)$ is coprime with each $x_j \in \mathbf{x}_0(\kappa^\prime)$. Indeed, assume that $f \in R$ is such that $f \ | \ x_i$ and $f \ | \ x_j$. Then $\mathcal{Q}(f)$ is invertible in $\tilde{R}[\tilde{\mathbf{x}}_0(\kappa)^{\pm 1}]$, and therefore, $f$ is a Laurent monomial in $\mathbf{x}_0(\kappa)$ with a coefficient from $\mathbb{K}^\times$. Applying $\mathcal{Q}^\prime$, we see that $f$ is also invertible in $\hat{R}[\mathbf{x}_0(\kappa^\prime)^{\pm 1}]$, and since the marked variables in $\hat{R}$ are prime and pairwise coprime, we conclude that $f \in \mathbb{K}^\times$. Therefore, $x_i$ is coprime with $x_j$, and by Proposition~\ref{p:prime_el}, we conclude that all the marked variables are prime in $R$.
\end{proof}

The next result corresponds to Proposition~\ref{p:birat_double_ufd}.
\begin{proposition*}
    Let $R$, $\tilde{R}$ and $\hat{R}$ be factorial Noetherian $\mathbb{K}$-algebras with no zero divisors and such that $R^{\times} = \tilde{R}^\times = \hat{R}^{\times} = \mathbb{K}^{\times}$, $(\mathcal{C},R)$, $(\tilde{\mathcal{C}},\tilde{R})$ and $(\hat{\mathcal{C}},\hat{R})$ be cluster structures of geometric type in $R$, $\tilde{R}$ and $\hat{R}$, with fixed initial extended clusters $\mathbf{x}_0$, $\tilde{\mathbf{x}}_0$ and $\hat{\mathbf{x}}_0$. Assume the following:
    \begin{enumerate}[1)]
        \item There are markers $\kappa$ and $\kappa^\prime$ that make $(\mathcal{C},\tilde{\mathcal{C}},\kappa)$ and $(\mathcal{C},\hat{\mathcal{C}},\kappa^\prime)$ related triples, as well as a pair of complementary birational quasi-isomorphisms $\mathcal{Q}:(\mathcal{C},R) \rightarrow (\tilde{\mathcal{C}},\tilde{R})$ and $\mathcal{Q}^\prime:(\mathcal{C},R)\rightarrow(\hat{\mathcal{C}},\hat{R})$; 
        \item The initial extended clusters $\tilde{\mathbf{x}}_0$ and $\hat{\mathbf{x}}_0$ are i-admissible;\label{ii:fbirat_ufd2}
        \item The initial extended cluster $\mathbf{x}_0$ is regular.
    \end{enumerate}
    Then $\mathbf{x}_0$ is i-admissible, as well as $\bar{\mathcal{A}}_{\mathbb{K}}(\hat{\mathcal{C}}) \subseteq \hat{R}$, $\bar{\mathcal{A}}_{\mathbb{K}}(\tilde{\mathcal{C}}) \subseteq \tilde{R}$ and $\bar{\mathcal{A}}_{\mathbb{K}}(\mathcal{C}) \subseteq R$. 
\end{proposition*}
\begin{proof}
    It follows from Condition~\ref{ii:fbirat_ufd2} and the Starfish lemma that $\bar{\mathcal{A}}_{\mathbb{K}}(\hat{\mathcal{C}}) \subseteq \hat{R}$ and $\bar{\mathcal{A}}_{\mathbb{K}}(\tilde{\mathcal{C}}) \subseteq \tilde{R}$. By Lemma~\ref{l:cv_fact_prim}, the marked variables in $R$ are irreducible and pairwise coprime. By Proposition~\ref{p:copmut}, each mutation $x^\prime$ of a marked variable $x$ is regular and coprime with $x$; hence, $\mathbf{x}_0(\kappa,\kappa^\prime)$ is i-admissible. It follows from Proposition~\ref{p:birat_double} that $\mathbf{x}_0$ is admissible and from Proposition~\ref{p:fbirat_single} that $\mathbf{x}_0$ is i-admissible and $\bar{\mathcal{A}}_{\mathbb{K}}(\mathcal{C}) \subseteq R$.
\end{proof}

\subsection{Proof of Proposition~\ref{p:upper_cont}}

We recall the statement of Proposition~\ref{p:upper_cont} with the full setup.
\begin{proposition*}
        Let $R$, $\tilde{R}$ and $\hat{R}$ be normal Noetherian domains, $(\mathcal{C},R)$, $(\tilde{\mathcal{C}},\tilde{R})$ and $(\hat{\mathcal{C}},\hat{R})$ be cluster structures of geometric type in $R$, $\tilde{R}$ and $\hat{R}$, with fixed initial extended seeds $(\mathbf{x}_0,B_0)$, $(\tilde{\mathbf{x}}_0,\tilde{B}_0)$ and $(\hat{\mathbf{x}}_0,\hat{B}_0)$.  Assume the following:
        \begin{enumerate}[1)]
            \item There are markers $\kappa$ and $\kappa^\prime$ that make $(\mathcal{C},\tilde{\mathcal{C}},\kappa)$ and $(\mathcal{C},\hat{\mathcal{C}},\kappa^\prime)$ related triples, as well as a pair of complementary birational quasi-isomorphisms $\mathcal{Q}:(\mathcal{C},R) \rightarrow (\tilde{\mathcal{C}},\tilde{R})$ and $\mathcal{Q}^\prime:(\mathcal{C},R)\rightarrow(\hat{\mathcal{C}},\hat{R})$;
            \item The matrices $\tilde{B}_0$, $\hat{B}_0$ and $B_0$ are of full rank;
            \item The inclusions $\bar{\mathcal{A}}(\tilde{\mathcal{C}}) \supseteq \tilde{R}$ and $\bar{\mathcal{A}}(\hat{\mathcal{C}}) \supseteq \hat{R}$ hold.
        \end{enumerate}
        Then $\bar{\mathcal{A}}(\mathcal{C}) \supseteq R$.
\end{proposition*}
\begin{proof}
    Since $\bar{\mathcal{A}}(\tilde{\mathcal{C}})\supseteq \tilde{R}$ holds, we see from equation~\eqref{eq:barincl} that, in particular, $\mathcal{L}(\mathbf{x}_i^\prime)\supseteq R$ for $i \in \mathcal{I}(\kappa^\prime)$. Since we also assume $\bar{\mathcal{A}}(\hat{\mathcal{C}})\supseteq \hat{R}$, it follows from Proposition~\ref{p:upper_cont1} applied to $(\mathcal{C},\hat{\mathcal{C}},\kappa^\prime)$ and $\mathcal{Q}^\prime$ that $\bar{\mathcal{A}}(\mathcal{C})\supseteq R$.
\end{proof}

%% file: example.tex
In this appendix, we  illustrate our theory in the case of cluster structures in\footnote{The ground field is not essential (one could choose any field extension of $\mathbb{Q}$).} $\mathbb{C}[\SL_3]$ compatible with Poisson brackets from the Belavin--Drinfeld\footnote{
Given a set of simple roots $\Delta$, a \emph{Belavin--Drinfeld triple} is a triple $\mathbf{\Gamma}:=(\Gamma_1,\Gamma_2,\gamma)$ where $\Gamma_1,\Gamma_2\subset \Delta$ and $\gamma : \Gamma_1 \rightarrow \Gamma_2$ is a nilpotent isometry. Nilpotency means that for every $\alpha \in \Gamma_1$ there is $k \geq 1$ such that $\gamma^k(\alpha) \notin \Gamma_1$. Belavin--Drinfeld triples (together with some additional continuous parameters) parameterize the moduli space of factorizable quasi-triangular Poisson brackets on a connected simple complex algebraic group $G$ \cite{bd2,etingof}. A more general class of Poisson homogeneous brackets is parameterized by pairs of Belavin--Drinfeld triples $(\mathbf{\Gamma}^r,\mathbf{\Gamma}^c)$ \cite{plethora}. This general context is not significant for our purpose, which is to provide explicit examples of birationally quasi-isomorphic cluster structures.}  class. A combinatorial recipe for constructing such cluster structures was first given in~\cite{plethora}. A general construction of quasi-isomorphisms for cluster structures in $\mathbb{C}[\SL_n]$ (in the context of Belavin--Drinfeld brackets) is available in~\cite{rho}. 

We consider four different cluster structures in $\mathbb{C}[\SL_3]$. Following our convention from Section~\ref{s:doublebirat}, we denote three of them as $\mathcal{C}$, $\tilde{\mathcal{C}}$ and $\hat{\mathcal{C}}$. For the fourth one, we use the bar notation $\bar{\mathcal{C}}$. Initial extended clusters consist of $8$ regular functions (or $9$ for $\mathbb{C}[\GL_3]$) indexed by pairs $(i,j)$, $i,j \in \{1,2,3\}$. In each case considered, the marker is defined in such a way as to send a variable with index $(i,j)$ to a variable with the same index $(i,j)$ (for instance, $\bar{f}_{ij}$ is related to $\tilde{f}_{ij}$). All initial quivers are illustrated in Figure~\ref{f:examples} along with the associated pairs of Belavin--Drinfeld triples.

The initial extended cluster for $\bar{\mathcal{C}}$ is well-known from the theory of total positivity. It is given by the following $8$ regular functions on $\SL_3(\mathbb{C})$:
\begin{equation}
    \bar{f}_{ij}(X) := \begin{cases}
        \det X_{[i,n-j+i]}^{[j,n]} &1 \leq i \leq j \leq 3;\\
        \det X_{[i,n]}^{[j,n-i+j]} &1 \leq j \leq i \leq 3.
    \end{cases}
\end{equation}
where $\det X_{I}^{J}$ for $I,J \subseteq [1,3]$ denotes the minor of a matrix $X \in \SL_3(\mathbb{C})$ on $I$ rows and $J$ columns. The variables $\bar{f}_{ij}(X)$ are regular, irreducible and pairwise coprime, and it is well-known that $\mathcal{A}_{\mathbb{C}}(\bar{\mathcal{C}}) = \mathbb{C}[\SL_3]$ (in fact, each matrix entry is a cluster variable).

The initial extended cluster for $\tilde{\mathcal{C}}$ is given by
\begin{equation}
    \tilde{f}_{ij}(X) = \begin{cases}
        \bar{f}_{ij}(X) &(i,j) \neq (1,3)\\
        \det \begin{bmatrix}
        x_{13} & x_{21}\\
        x_{23} & x_{31}
        \end{bmatrix} &(i,j) = (1,3).
    \end{cases}
\end{equation}
The pair $(\tilde{\mathcal{C}},\bar{\mathcal{C}})$ is related, with the marked variables given by $\bar{f}_{31}$ and $\tilde{f}_{31}$. Define a rational map $\mathcal{U} : \SL_3(\mathbb{C}) \dashrightarrow \SL_3(\mathbb{C})$ by
\begin{equation}
    \mathcal{U}(X) := \left(I + \frac{x_{21}}{x_{31}}e_{12}\right)\cdot X.
\end{equation}
The map $\mathcal{U}$ is birational, with the inverse given by
\begin{equation}
    \mathcal{U}^{-1}(X) = \left(I-\frac{x_{21}}{x_{31}}e_{12}\right)\cdot X.
\end{equation}
The induced map on localizations $\mathcal{U}^* : \mathbb{C}[\SL_3][x_{31}^{\pm 1}] \rightarrow \mathbb{C}[\SL_3][x_{31}^{\pm 1}]$ is a birational quasi-isomorphism. According to Proposition~\ref{p:fbirat_single}, the initial cluster of $\tilde{\mathcal{C}}$  consists of irreducible pairwise coprime regular functions, and in order to claim $\bar{\mathcal{A}}_{\mathbb{C}}(\tilde{\mathcal{C}}) = \mathbb{C}[\SL_3]$, one only needs to verify that 
\begin{itemize}
    \item The mutation $\tilde{f}_{31}^\prime$ is regular and not divisible by $x_{31}$;
    \item Setting $F_{31}:=\{\tilde{f}_{ij} \ | \ (i,j) \neq (3,1)\} \cup \{\tilde{f}_{31}^\prime\}$, any regular function $p \in \mathbb{C}[\SL_3]$ is an element of $\mathcal{L}_{\mathbb{C}}(F_{31})$; that is, a Laurent polynomial in terms of the functions in $F_{31}$.
\end{itemize}
For the first condition, one can see that
\begin{equation}
    f_{31}^\prime = \det \begin{bmatrix}
        x_{13} & x_{22}\\
        x_{23} & x_{32}
    \end{bmatrix},
\end{equation}
while the second condition requires some work (which was done in \cite{plethora} for this set of examples).


\begin{figure}[t]
\centering
 \begin{subfigure}[t]{3in}
 \centering
 \includegraphics[scale=0.2]{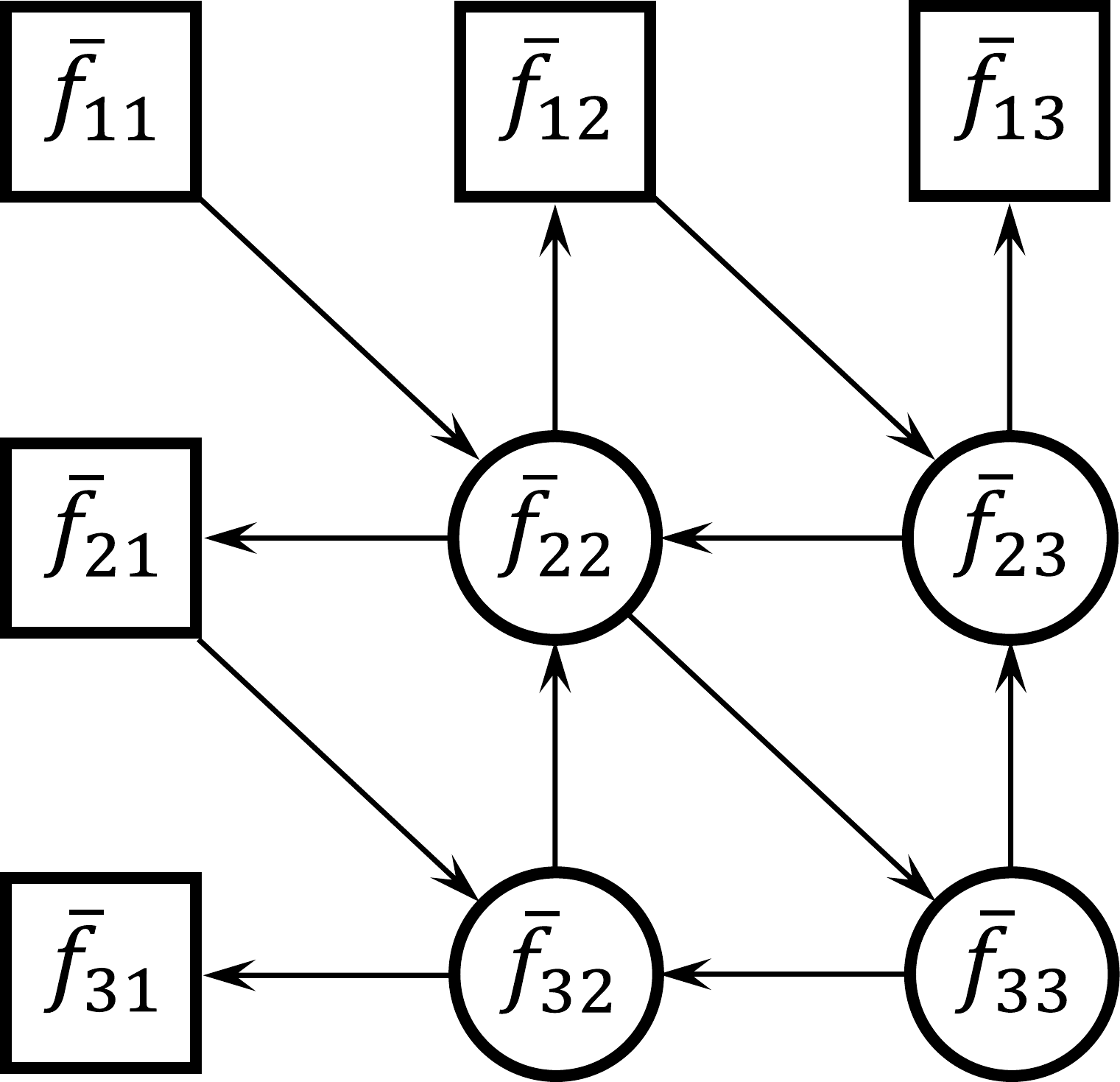}
 \subcaption{Case of $\bar{\mc}$, $(\bg^r,\bg^c) = (\bg_\std,\bg_\std)$.}
 \label{f:st}
 \end{subfigure}
 \begin{subfigure}[t]{3in}
 \centering
 \includegraphics[scale=0.2]{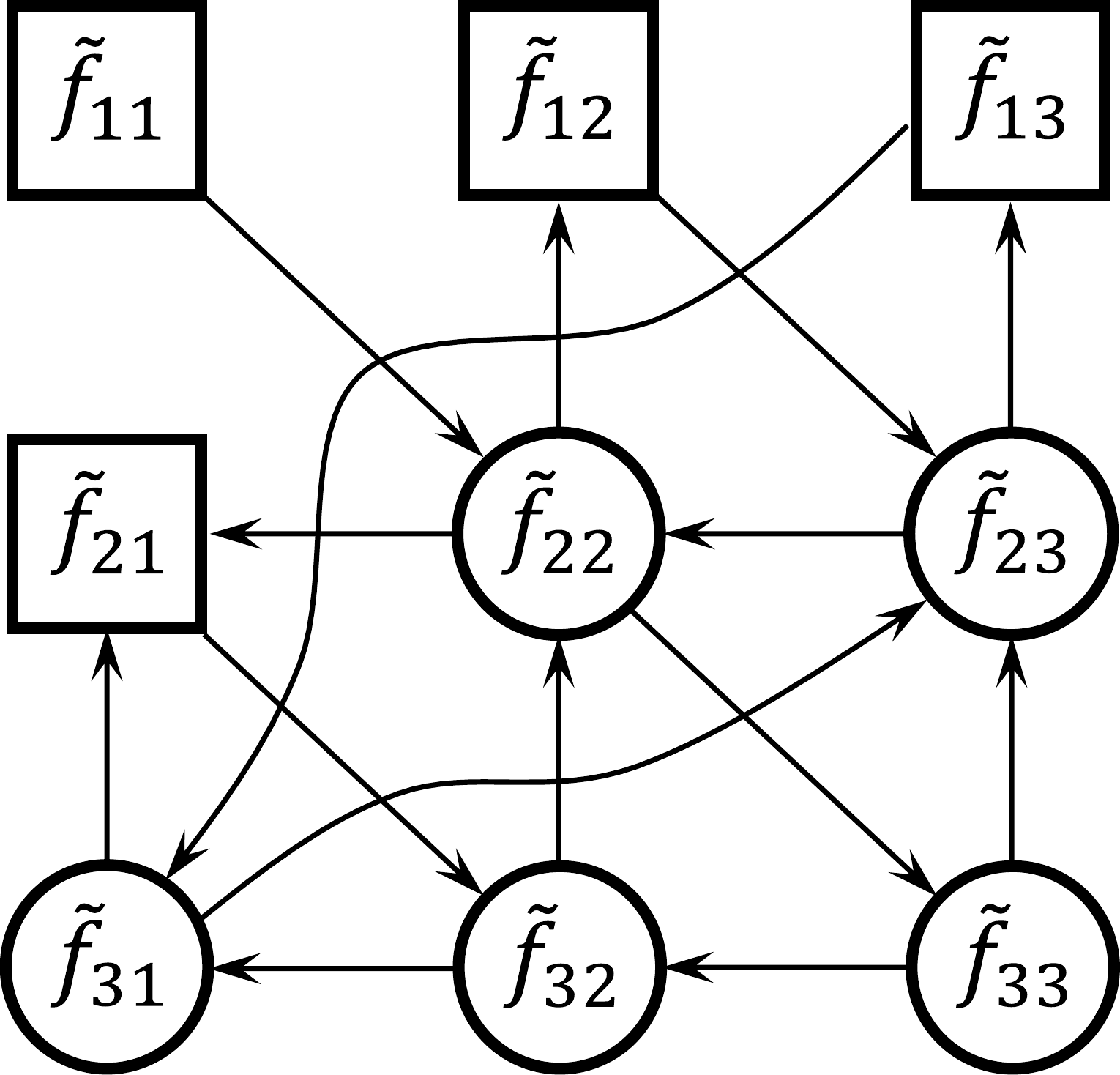}
 \subcaption{Case of $\tilde{\mc}$, $(\bg^r,\bg^c) = (\bg,\bg_\std)$.}
 \label{f:tilde}
 \end{subfigure}
 \begin{subfigure}[t]{3in}
 \centering
\vspace{4mm}
 \includegraphics[scale=0.2]{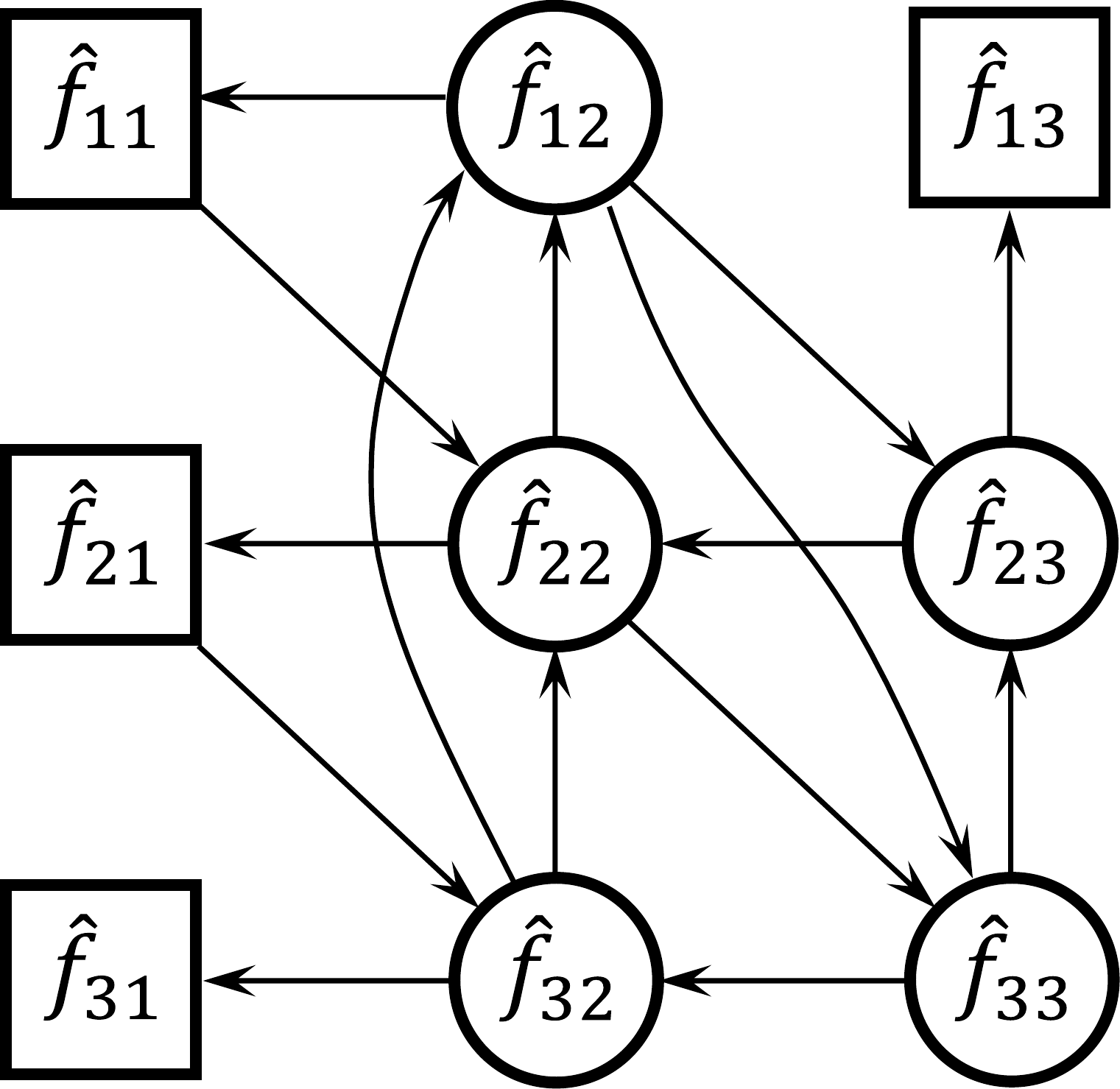}
 \subcaption{Case of $\hat{\mc}$, $(\bg^r,\bg^c) = (\bg_\std,\bg)$.}
 \label{f:hat}
 \end{subfigure}
 \begin{subfigure}[t]{3in}
 \centering
\vspace{4mm}
 \includegraphics[scale=0.2]{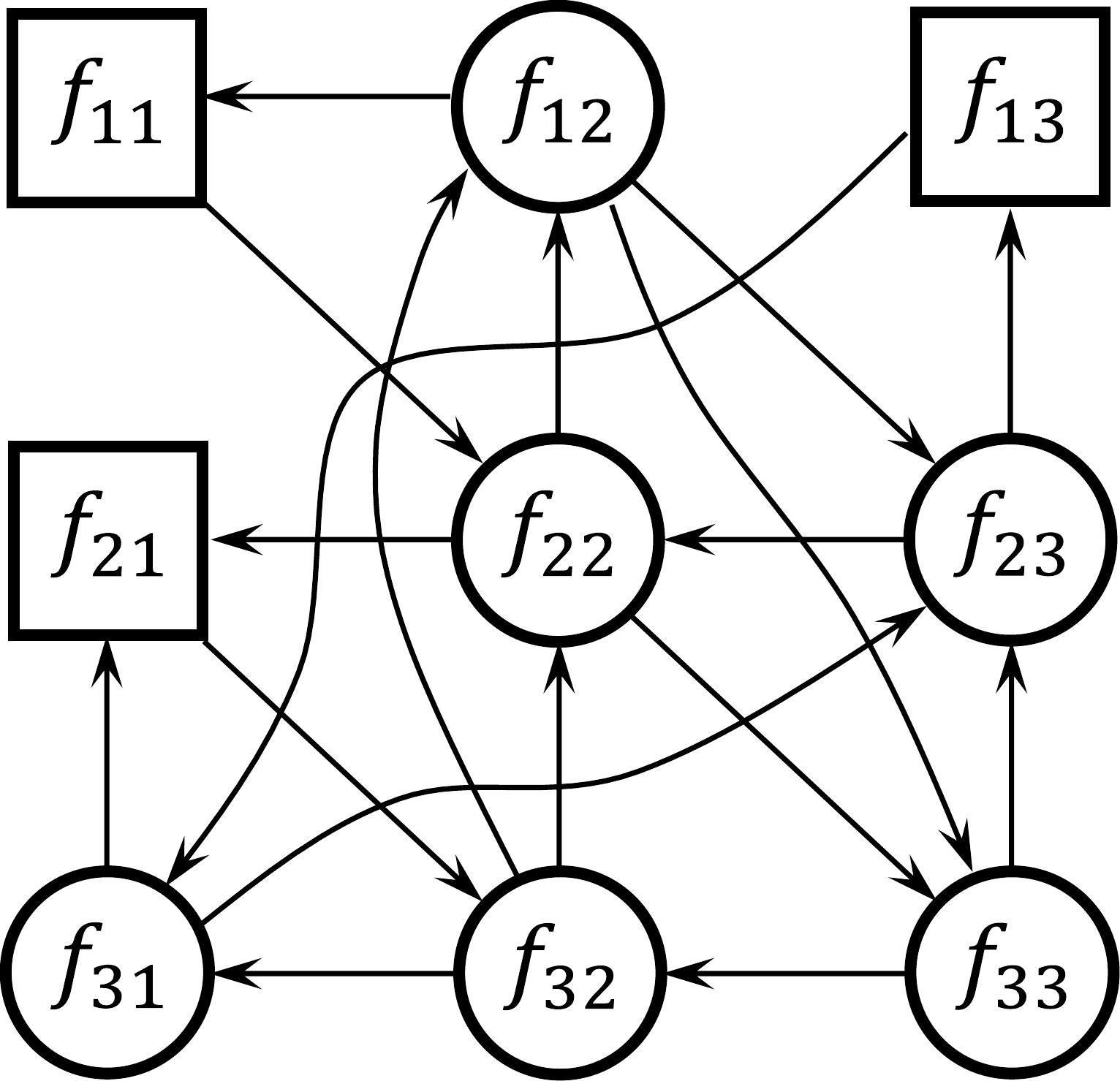}
 \subcaption{Case of $\mc$, $(\bg^r,\bg^c) = (\bg,\bg)$.}
  \label{f:final}
 \end{subfigure}
 \caption{Initial quivers for $\bar{\mc}$, $\tilde{\mc}$, $\hat{\mc}$ and $\mc$; here, $\bg = (\{2\},\{1\},2 \mapsto 1)$ and $\bg_{\std} = (\emptyset,\emptyset,\emptyset\rightarrow \emptyset)$.}
 \label{f:examples}
 \end{figure}
The initial extended cluster for $\hat{\mathcal{C}}$ consists of the following functions:
\begin{equation}
    \hat{f}_{ij}(X) := \begin{cases}
        \bar{f}_{ij}(X) &(i,j) \notin \{(2,1), (3,2)\};\\
         \det X_{[2,3]}^{[1,2]} \det X_{[1,2]}^{[2,3]} - \det X_{[2,3]}^{\{1,3\}} \det X_{[1,2]}^{\{1,3\}} &(i,j) = (2,1);\\
         x_{32}\det X_{[1,2]}^{[2,3]} - x_{33} \det X_{[1,2]}^{\{1,3\}} &(i,j) = (3,2).
    \end{cases}
\end{equation}
The pair $(\hat{\mathcal{C}},\bar{\mathcal{C}})$ is related, with the marked variables given by $\hat{f}_{12}$ and $\bar{f}_{12}$. Define a birational map $\mathcal{U}^\prime : \SL_3(\mathbb{C}) \dashrightarrow \SL_3(\mathbb{C})$ via 
\begin{equation}
    \mathcal{U}^\prime(X) := X\cdot \left(I + \frac{\det X_{[1,2]}^{\{1,3\}}}{\det X_{[1,2]}^{[2,3]} }e_{32}\right).
\end{equation}
Then the induced map $(\mathcal{U}^\prime)^*: \mathbb{C}[\SL_3][\hat{f}_{12}(X)^{\pm 1}] \rightarrow \mathbb{C}[\SL_3][\bar{f}_{12}(X)^{\pm 1}]$ is a birational quasi-isomorphism.

The initial extended cluster for $\mathcal{C}$ is given by
\begin{equation}
    f_{ij}(X) := \begin{cases}
        \bar{f}_{ij}(X) &(i,j) \notin \{(2,1),(3,2),(1,3)\};\\
        \hat{f}_{ij}(X) &(i,j) \in \{(2,1),(3,2)\};\\
        \tilde{f}_{13}(X) &(i,j) = (1,3).
    \end{cases}
\end{equation}
The pair $(\mathcal{C},\bar{\mathcal{C}})$ is related, with the marked variables given by $\{f_{12},f_{31}\}$ in $\mathcal{C}$ and $\{\bar{f}_{12},\bar{f}_{31}\}$ in $\bar{\mathcal{C}}$. There is a birational quasi-isomorphism $(\mathcal{U}^{\pprime})^* : \mathbb{C}[\SL_3][f_{12}^{\pm 1},f_{31}^{\pm 1}]\rightarrow \mathbb{C}[\SL_3][\bar{f}_{12}^{\pm 1},\bar{f}_{31}^{\pm 1}]$ which, as a birational map $\mathcal{U}^{\pprime} : \SL_3(\mathbb{C})\dashrightarrow \SL_3(\mathbb{C})$, is given by
\begin{equation}
\mathcal{U}^{\pprime}(X) = \left(I + \frac{x_{21}}{x_{31}}e_{12}\right)\cdot X\cdot \left(I + \frac{\det X_{[1,2]}^{\{1,3\}}}{\det X_{[1,2]}^{[2,3]} }e_{32}\right).
\end{equation}
In this set of examples, one can also see that $(\mathcal{C},\tilde{\mathcal{C}})$ and $(\mathcal{C},\hat{\mathcal{C}})$ are related pairs. The map $(\mathcal{U}^\prime)^*$ is a birational quasi-isomorphism for the first pair, and the map $(\mathcal{U})^*$ is a birational quasi-isomorphism for the second pair. The two maps constitute a pair of complementary birational quasi-isomorphisms. According to Proposition~\ref{p:birat_double_ufd}, if the conditions of the Starfish lemma and Upper bounds hold for $\tilde{\mathcal{C}}$ and $\hat{\mathcal{C}}$, they hold for $\mathcal{C}$ as well. This was proved in~\cite{plethora} by more direct methods.

%% file: examplegr.tex
Consider the affine cone $\widehat{\Gr}(n,m)$ over the Grassmannian $\Gr(n,m)$. It is an affine variety whose coordinate ring is given by the homogeneous coordinate ring of $\Gr(n,m)$. Points of $\widehat{\Gr}(n,m)$ can be identified with decomposable $k$-forms $v_1 \wedge v_2 \wedge \cdots \wedge v_k$, $v_i \in \mathbb{C}^n$. 

Given a natural action of $\SL_n(\mathbb{C})$ upon $\mathbb{C}^n$, extend it to the diagonal action upon $(\mathbb{C}^n)^{\times m} :=\mathbb{C}^n \times \mathbb{C}^n \times \cdots \times \mathbb{C}^n$ ($m$-copies). The coordinate ring $\mathbb{C}[\widehat{\Gr}(n,m)]$ can be described as the ring of invariants
\begin{equation}
    \mathbb{C}[\widehat{\Gr}(n,m)] = \mathbb{C}[(\mathbb{C}^n)^{\times m}]^{\SL_n(\mathbb{C})}.
\end{equation}
It is generated by Pl\"ucker coordinates, which can be described as follows. Identify $(\mathbb{C}^n)^{\times m}$ with the space $\Mat_{\mathbb{C}}(n,m)$ of $n\times m$ matrices over $\mathbb{C}$, and let $Z:=(z_{ij})$ be an $n\times m$ matrix of coordinate functions on $\Mat_{\mathbb{C}}(n,m)$. The diagonal action of $\SL_n(\mathbb{C})$ upon $(\mathbb{C}^n)^{\times m}$ becomes the left multiplication of $\SL_n(\mathbb{C})$ upon $\Mat_{\mathbb{C}}(n,m)$. For an ordered tuple $J$ in the indices $\{1,\ldots,n\}$, the Pl\"ucker coordinate $P_J$ is given by 
    $P_J := Z^{J}$;
that it, $P_J$ is a minor of $Z$ on $J$ columns and all rows.

In \cite{scott}, it was shown that there exists a cluster structure $\tilde{\mathcal{C}}$ such that $\mathbb{C}[\widehat{\Gr}(n,m)] = \mathcal{A}_{\mathbb{C}}(\tilde{\mathcal{C}})$ (see also \cite{fomin6} for a textbook treatment). For $n=3$ and $m = 7$, the cluster structure $\tilde{\mathcal{C}}$ is of rank $6$, with $7$ frozen variables (the initial seed is illustrated in Figure~\ref{f:examplegr}). To illustrate our theory, we will exhibit another cluster structure $\mathcal{C}$ such that $\mathbb{C}[\widehat{\Gr}(n,m)] = \bar{\mathcal{A}}_{\mathbb{C}}(\mathcal{C})$, whose rank is $7$, with $6$ frozen variables, as well as a birational quasi-isomorphism between $\tilde{\mathcal{C}}$ and $\mathcal{C}$ (we will provide a general treatment of cluster structures in $\mathbb{C}[\widehat{\Gr}(n,m)]$ of higher ranks in a separate paper).
 
Let us define the following variables:
\begin{equation}
    S_{i67} := \det \begin{bmatrix}
        P_{i67} & P_{i57}\\ P_{457} & P_{456}
    \end{bmatrix}, \ \ i \in \{1,2,3,4\}.
\end{equation}
The initial seed of another cluster structure $\mathcal{C}$ in $\mathbb{C}[\widehat{\Gr}(n,m)]$ is illustrated in Figure~\ref{f:examplegrnstd}. 

Let us introduce a birational map $\mathcal{U} : \widehat{\Gr}(n,m)\dashrightarrow \widehat{\Gr}(n,m)$. In the nonzero locus of $P_{567}$, let $X = (x_{ij})$ be a $3\times 4$ matrix of coordinate functions given by
\begin{equation}
    (Z^{[5,7]})^{-1}Z = \begin{bmatrix} X & I_3\end{bmatrix},
\end{equation}
where $I_3$ is a $3 \times 3$ identity matrix.
In these coordinates, $\mathcal{U}$ is given by
\begin{equation}
    \mathcal{U}(X) = \left(I_3 + \frac{x_{24}}{x_{34}}e_{12}\right)X.
\end{equation}
Note that $P_{457}(Z) = -x_{24}P_{567}(Z)$ and $P_{456}(Z) = x_{34}P_{567}(Z)$.
One can see that $\mathcal{U}$ induces an isomorphism of localizations \begin{equation}
\mathcal{U}^* : \mathbb{C}[\widehat{\Gr}(n,m)][P_{456}^{\pm 1}] \xrightarrow{\sim} \mathbb{C}[\widehat{\Gr}(n,m)][P_{456}^{\pm 1}].
\end{equation}
It is a birational quasi-isomorphism between  $\tilde{\mathcal{C}}$ and $\mathcal{C}$ (with fixed initial seeds). The only marked variable is given by $P_{456}$.

\begin{figure}[t]
\centering
 \includegraphics[scale=0.2]{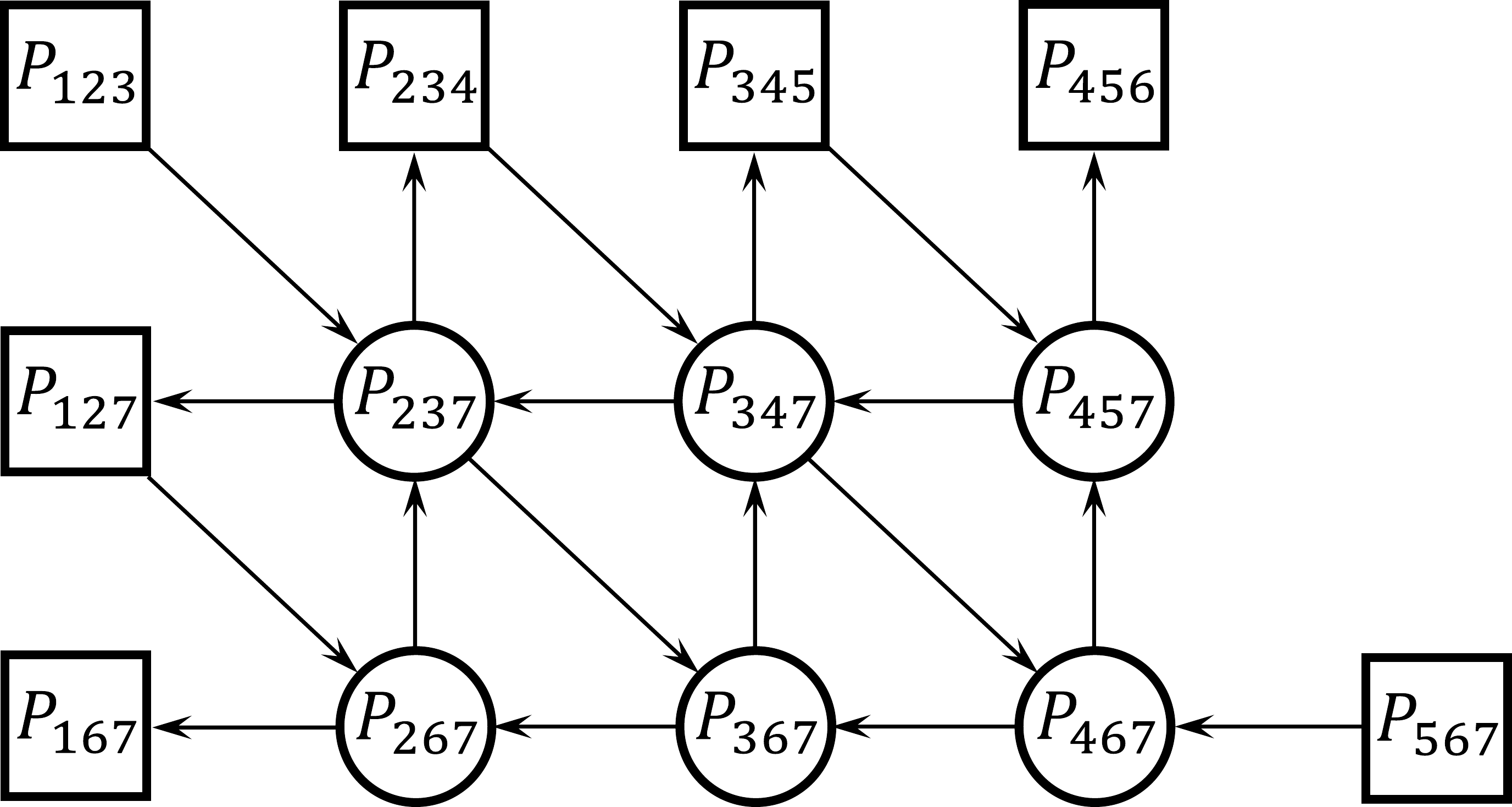}
 \caption{Initial seed for cluster structure $\tilde{\mathcal{C}}$ in $\mathbb{C}[\widehat{\Gr}(3,7)]$.}
 \label{f:examplegr}
 \end{figure}

Since $\mathbb{C}[\widehat{\Gr}(n,m)]$ is a UFD, we can apply Proposition~\ref{p:fbirat_single}. We see that the mutation of the marked variable is given by $P_{456}^\prime = P_{467}$. Clearly, the mutation is regular and coprime with $P_{456}$; therefore, $\bar{\mathcal{A}}_{\mathbb{C}}(\mathcal{C}) \subseteq \mathbb{C}[\widehat{\Gr}(n,m)]$. 


\begin{figure}[t]
\centering
 \includegraphics[scale=0.2]{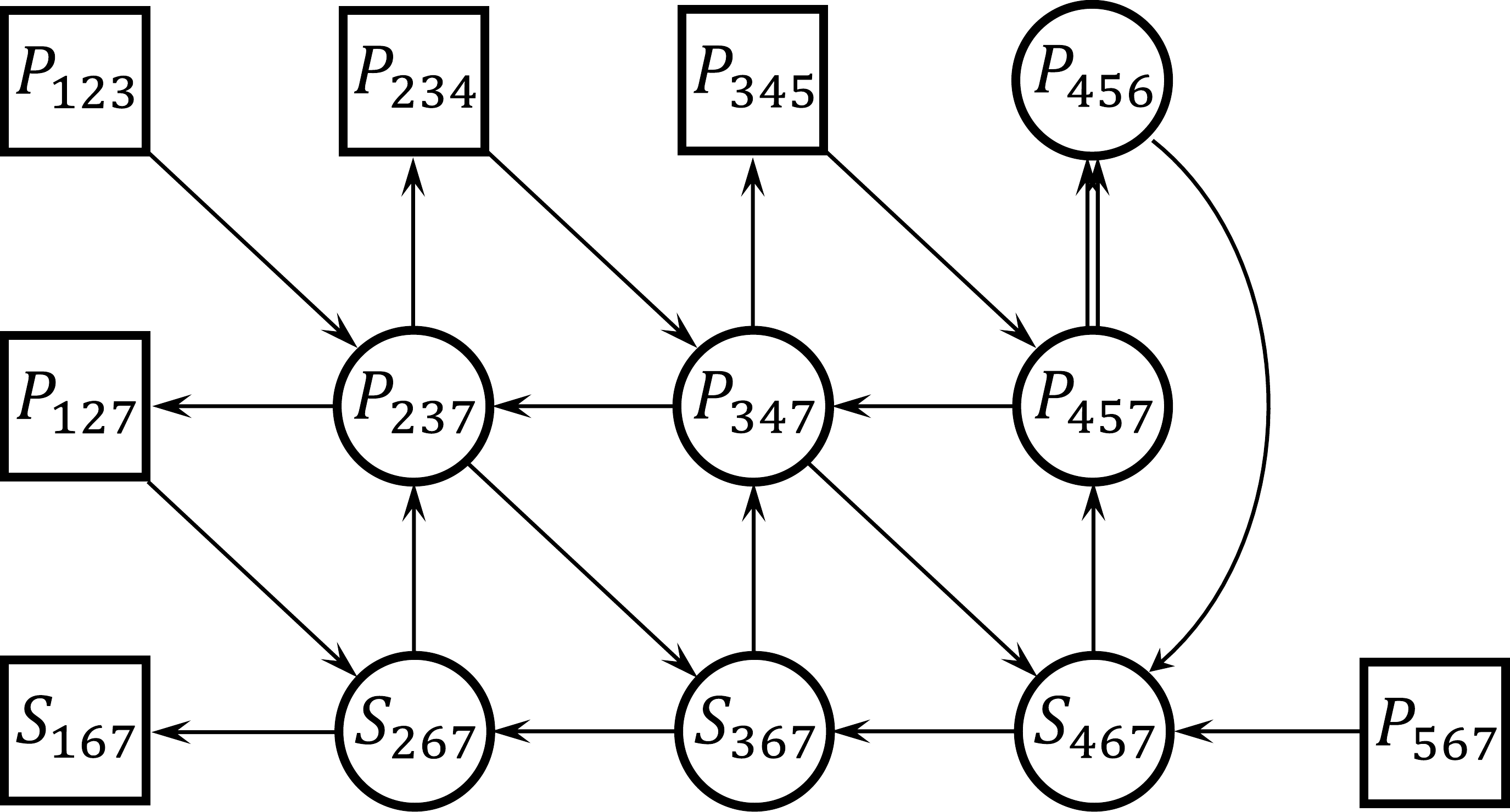}
 \caption{Initial seed for cluster structure ${\mathcal{C}}$ in $\mathbb{C}[\widehat{\Gr}(3,7)]$.}
 \label{f:examplegrnstd}
 \end{figure}